\documentclass{amsart}

\usepackage{amsfonts}
\usepackage{amsmath}
\usepackage{amsthm}
\usepackage{amssymb}
\usepackage{ifsym}
\usepackage{dirtytalk}

\theoremstyle{definition}

\theoremstyle{remark}

\numberwithin{equation}{section}

\newcommand\blfootnote[1]{%
  \begingroup
  \renewcommand\thefootnote{}\footnote{#1}%
  \addtocounter{footnote}{-1}%
  \endgroup
}

\begin{document}

\title[The Cube Property for the Class of Linear Orders]{Every Linear Order Isomorphic to its Cube is Isomorphic to its Square}
\author{Garrett Ervin}

\begin{abstract}
In 1958, Sierpi\'nski asked whether there exists a linear order $X$ that is isomorphic to its lexicographically ordered cube but is not isomorphic to its square. The main result of this paper is that the answer is negative. More generally, if $X$ is isomorphic to any one of its finite powers $X^n$, $n>1$, it is isomorphic to all of them. 
\end{abstract}

\maketitle

\section{Introduction}

\blfootnote{\emph{Keywords:} linear order, lexicographical product, cube problem, cube property.}
\blfootnote{I am indebted to Justin Moore for carefully reading an earlier version of this paper and offering a number of useful suggestions.}

Suppose that ($\mathfrak{C}, \times)$ is a class of structures equipped with a product. In many cases it is possible to find an infinite structure $X \in \mathfrak{C}$ that is isomorphic to its own square. If $X^2 \cong X$, then by multiplying again we have $X^3 \cong X$. Determining whether the converse holds for a given class $\mathfrak{C}$, that is, whether $X^3 \cong X$ implies $X^2 \cong X$ for all $X \in \mathfrak{C}$, is called the \emph{cube problem} for $\mathfrak{C}$. If the cube problem for $\mathfrak{C}$ has a positive answer, then $\mathfrak{C}$ is said to have the \emph{cube property}. 

The cube problem is related to three other basic problems concerning the multiplication of structures in a given class $(\mathfrak{C}, \times)$. 
\begin{enumerate}
\item[1.] Does $A \times Y \cong X$ and $B \times X \cong Y$ imply $X \cong Y$ for all $A, B, X, Y \in \mathfrak{C}$? Equivalently, does $A \times B \times X \cong X$ imply $B \times X \cong X$ for all $A, B, X \in \mathfrak{C}$?
\item[2.] Does $X^2 \cong Y^2$ imply $X \cong Y$ for all $X, Y \in \mathfrak{C}$?
\item[3.] Does $A \times Y \cong X$ and $A \times X \cong Y$ imply $X \cong Y$ for all $A, X, Y \in \mathfrak{C}$? Equivalently, does $A^2 \times X \cong X$ imply $A \times X \cong X$ for all $A, X \in \mathfrak{C}$?
\end{enumerate}
The first question is called the \emph{Schroeder-Bernstein problem} for $\mathfrak{C}$, and if it has a positive answer, then $\mathfrak{C}$ is said to have the \emph{Schroeder-Bernstein property}. The second question is called the \emph{unique square root problem} for $\mathfrak{C}$, and if its answer is positive, then $\mathfrak{C}$ has the \emph{unique square root property}. Taken together, the first two questions are sometimes called the \emph{Kaplansky test problems}, after Irving Kaplansky who posed them in \cite{Kaplansky} as a heuristic test for whether a given class of abelian groups (under the direct product) has a satisfactory structure theory (``I believe their defeat is convincing evidence that no reasonable invariants exist"). Tarski \cite{Tarski} had posed them previously for the class of Boolean algebras. All three questions are listed in Hanf's seminal paper \cite{Hanf} on products of Boolean algebras. We will refer to Question 3 as the \emph{weak Schroeder-Bernstein problem} for $\mathfrak{C}$, and the corresponding property as the \emph{weak Schroeder-Bernstein property}.

A negative solution to Question 3 obviously gives a negative solution to Question 1. If the product for $\mathfrak{C}$ is commutative, it gives a negative solution to Question 2 as well. If the cube problem for $\mathfrak{C}$ has a negative solution, that is, if there is an $X \in \mathfrak{C}$ that is isomorphic to its cube but not to its square, then all three questions have a negative solution, without assuming commutativity of the product. In practice, it is often by constructing such an $X$ that these three problems are solved. 

If the class $\mathfrak{C}$ does not contain any infinite structure isomorphic to its cube, then the cube property holds trivially. When the cube property does not hold trivially, it usually fails. The first result in this direction is due to Hanf, who constructed in \cite{Hanf} a Boolean algebra that is isomorphic to its cube but not its square. Tarski \cite{Tarski2} and J\'onsson \cite{Jonsson} immediately adapted Hanf's result to show the failure of the cube property for the class of semigroups, the class of groups, the class of rings, and various other classes of algebraic structures. Hanf's example, and consequently many of those produced by Tarski and J\'onsson, is of size continuum, and for some time it was open whether there were countable examples witnessing the failure of the cube property for these various classes. 

In 1965, Corner showed in \cite{Corner} that indeed there exists a countable (torsion-free, abelian) group $G$ isomorphic to $G^3$ but not $G^2$. Later, Jones \cite{Jones} showed that it is even possible construct a finitely generated (necessarily non-abelian) group isomorphic to its cube but not its square. In 1979, Ketonen \cite{Ketonen} solved the so-called Tarski cube problem by producing a countable Boolean algebra isomorphic to its cube but not its square. 

Throughout the 1970s and 1980s, Trnkov\'a solved the cube problem negatively for many different classes of topological spaces and relational structures, including the class of graphs under several different notions of graph product \cite{TrnkGraph}. Her topological results are summarized in \cite{Trnk3}. Answering a question of Trnkov\'a, Orsatti and Rodino showed in \cite{Orsatti} that there is even a \emph{connected} topological space homeomorphic to its cube but not its square. Koubek, Ne\v set\v ril, and R\"odl \cite{Koubek} showed that the cube property fails for the class of partial orders, as well as for other classes of relational structures. More recently, Eklof and Shelah \cite{EkShel} constructed an $\aleph_1$-separable group isomorphic to its cube but not its square, and Gowers \cite{Gowers} constructed a Banach space linearly homeomorphic to its cube but not its square. 

On the other hand, there are rare instances when the cube property holds nontrivially. It holds for the class of sets under the cartesian product, since any set in bijective correspondence with its cube is either infinite, empty, or a singleton, and hence in bijection with its square. This is immediate if one assumes the axiom of choice, but it can be proved without the axiom of choice using the Schroeder-Bernstein theorem. Similarly easily, the cube property holds for the class of vector spaces (over a fixed field) under the direct product. Less trivially, the cube property holds for the class of countably complete Boolean algebras. This follows from the Schroeder-Bernstein theorem for such algebras. Trnkov\'a \cite{Trnk2} showed that the cube property also holds for the class of countable metric spaces (where isomorphism means homeomorphism), as well as for closed subspaces of Cantor space \cite{Trnk5}. Koubek, Ne\v set\v ril, and R\"odl showed in \cite{Koubek} that the cube property holds for the class of equivalence relations. It is worth noting that for all of these classes, it is actually possible to establish the stronger Schroeder-Bernstein property. 

In his 1958 book \emph{Cardinal and Ordinal Numbers} \cite{Sierpinski}, Sierpi\'nski posed the cube problem (although he does not use the term) for the class $(LO, \times_{lex})$ of linear orders under the lexicographical product. On page 232, he writes,
\begin{center}``We do not know so far $\ldots$ any type $\alpha$ such that $\alpha = \alpha^3 \neq \alpha^2$." \end{center}
Here, ``type" means linear order type, and the ordering on the cartesian powers $\alpha^2$ and $\alpha^3$ is the lexicographical ordering\footnote{Sierpi\'nski actually ordered these powers anti-lexicographically, though in this paper we will use the lexicographical ordering. This does not change the problem.}. Although the cube problem has been solved for many other classes of structures, Sierpi\'nski's question has remained open. One major difference in this version of the cube problem is that, unlike the products for the other classes so far discussed, the lexicographical product of linear orders is not commutative. Though he does not make a conjecture in his book, his language suggests that Sierpi\'nski expected that such an $\alpha$ exists, that is, that the cube property fails for $(LO, \times_{lex})$. He was already aware of examples of linear orders witnessing the failure of the unique square root property and (the right-sided and left-sided versions of) the Schroeder-Bernstein property.

The main result of this paper is that in fact the cube property holds for $(LO, \times_{lex})$. This appears as Theorem \ref{mainthm} below.

\theoremstyle{definition}
\newtheorem*{mt}{Main Theorem}
\begin{mt}
If $X$ is a linear order and $X^3 \cong X$, then $X^2 \cong X$. More generally, if $X^n \cong X$ for some $n > 1$, then $X^2 \cong X$.
\end{mt}

Thus the cube property holds for the class of linear orders despite the fact that the Schroeder-Bernstein property and unique square root property fail. We will show in Section 6 that even the weak Schroeder-Bernstein property fails for $(LO, \times_{lex})$. In this sense, the cube property is closer to failing for $(LO, \times_{lex})$ than it is for the other classes for which it is known to hold. 

In the remainder of this paper we prove the main theorem as well as some related results. In Section 2, we define the necessary notation and terminology, and give some examples of countable linear orders isomorphic to their own squares. We also give the easy proof that the cube property holds for linear orders with both a top and bottom point. In Section 3, we prove a representation theorem for linear orders $X$ that are invariant under left multiplication by a fixed order $A$, that is, that satisfy the isomorphism $A \times X \cong X$. It turns out that such orders can essentially be represented as unions of tail-equivalence classes in the lexicographically ordered tree $A^{\omega}$. This theorem can be generalized to characterize the isomorphism $A^n \times X \cong X$ for any $n\geq1$, and in particular $A^2 \times X \cong X$. It is then possible to write down a sufficient condition, namely the existence of a \emph{parity-reversing automorphism} of $A^{\omega}$, for the implication $A^2 \times X \cong X \implies A \times X \cong X$ to hold for every $X$. 

The combinatorial heart of the proof is contained in Section 4, where parity-reversing automorphisms of $A^{\omega}$ are constructed for various orders $A$, including all countable orders. These maps are built using generalized versions of the classical Schroeder-Bernstein bijection. 

In this context, the isomorphism $X^3 \cong X$ can be rewritten $X^2 \times X \cong X$. The results of Sections 3 and 4 then give that if $X^{\omega}$ has a parity-reversing automorphism, then we must also have $X \times X \cong X$ (i.e. $X^2 \cong X$). We will show at the end of Section 4 that if $X^3 \cong X$, then indeed $X^{\omega}$ has a parity-reversing automorphism, completing the proof of the cube property. The proof is then generalized to give $X^n \cong X \implies X^2 \cong X$ for any linear order $X$ and $n > 1$. 

In Section 5, we show that the main theorem is not vacuous by illustrating a general way of constructing orders $X$ such that $X^n \cong X$ for a fixed $n>1$. Such orders can be arranged to be of any cardinality. In Section 6 we will show that there exists an order $A$ such that $A^{\omega}$ does not have a parity-reversing automorphism, and as a consequence we will be able to construct a counterexample to the weak Schroeder-Bernstein property for $(LO, \times_{lex})$. In Section 7, several related problems concerning the multiplication of linear orders are discussed. 

The paper is for the most part self-contained. General background on linear orders can be found in \cite{Rosenstein}. 

\section{Preliminaries}

\subsection{Terminology}
A \emph{linear order} is a pair $(X, <_X)$ where $X$ is a set and $<_X$ is a binary relation that totally orders $X$. We will always refer to linear orders by their underlying sets, and write $<$ without any subscript. Throughout the paper, ``order" always means linear order, and ``isomorphism" means order isomorphism. 

Given a linear order $X$, a subset $I \subseteq X$ is called an \emph{interval} if for all points $x, y, z \in X$, if $x < y < z$ and $x, z \in I$, then $y \in I$. Every singleton is an interval, as is $X$ itself. Given points $x, y \in X$ with $x < y$, the interval notation $(x, y)$, $[x, y)$, $(x, y]$, and $[x, y]$ has its usual meaning. If $I$ and $J$ are intervals in $X$, and for all $x \in I$ and $y \in J$ we have $x<y$, then $I$ lies to the left of $J$ and we write $I<J$. An interval $I$ is called an \emph{initial segment} of $X$ if whenever $x \in I$ and $y< x$ then $y \in I$. An interval $J$ is called a \emph{final segment} of $X$ if $J$ is the complement of an initial segment of $X$. 

An order (or interval) may have endpoints. The terms minimal element, left endpoint, and bottom point will be used interchangeably, as will maximal element, right endpoint, and top point. 

An order $X$ is \emph{dense} if between any two distinct points in $X$ one may find a third that lies strictly between them. A subset $D \subseteq X$ is \emph{dense in $X$} if for any two points in $X$, either one of them lies in $D$ or there exists a point between them that lies in $D$. An order $X$ is \emph{complete} if every bounded monotonic sequence in $X$ converges to a point in $X$. 

For $X$ a linear order, $X^*$ denotes the reverse order. That is, $X$ and $X^*$ share the same underlying set of points, but $x<y$ in $X$ if and only if $y<x$ in $X^*$. To every function $f: X \rightarrow X$, there is a corresponding function on $X^*$, denoted $f^*$, which acts identically to $f$ on the underlying set of points shared by $X$ and $X^*$. If $f$ is an order automorphism of $X$, then $f^*$ is an order automorphism of $X^*$. 

If $X$ and $Y$ are linear orders, the \emph{lexicographical product} $X \times Y$ is the order obtained by lexicographically ordering the cartesian product of $X$ and $Y$. That is, $X \times Y = \{(x, y): x \in X, y \in Y\}$ ordered by the rule $(x_0, y_0) < (x_1, y_1)$ if and only if $x_0 < x_1$ (in $X$), or $x_0 = x_1$ and $y_0 < y_1$ (in $Y$). We will usually suppress the product symbol, writing $XY$ instead of $X \times Y$. 

Visually, $XY$ is the order obtained by replacing every point in $X$ with a copy of $Y$. Every point $x \in X$ determines an interval of points in $XY$ of order type $Y$, namely the set of pairs $(x, \cdot)$ with left entry $x$. One may also visualize $XY$ as a tree with two ordered levels. The first level has $X$-many nodes, and each of these has $Y$-many descendants. The order type of the terminal nodes is $XY$.

The lexicographical product is associative, in the sense that $(X \times Y) \times Z$ is isomorphic to $X \times (Y \times Z)$ for all orders $X, Y, Z$. But it is not commutative. For example, let $\mathbb{Z}$ be the integers in their usual order, and let $2$ be the unique linear order with two elements. Then $\mathbb{Z}2$ is isomorphic to $\mathbb{Z}$, but $2\mathbb{Z}$ is not, as the latter order contains a bounded infinite increasing sequence, whereas every infinite increasing sequence in $\mathbb{Z}$ is unbounded. 

Since the lexicographical product is associative, we identify the $n$-length product $X_0 X_1\ldots X_{n-1}$ with the set of $n$-tuples $\{(x_0, x_1, \ldots, x_{n-1}): x_i \in X_i, i < n\}$ ordered lexicographically. If $X_i = X$ for all $i < n$, this order is abbreviated as $X^n$. We also define the $\omega$-length product $X_0X_1\ldots$ as the set of sequences $\{(x_0, x_1, \ldots): x_i \in X_i, i \in \omega\}$ ordered lexicographically. If $X_i = X$ for all $i \in \omega$, this order is abbreviated as $X^{\omega}$. One may think of $X^n$ as a tree with $n$ levels and $X^{\omega}$ as a tree with $\omega$-many levels, but since in the latter case there are no terminal nodes, one must lexicographically order the branches of the tree to recover the order.

The replacement operation will also play an important role. Given an order $X$, and for every point $x \in X$ an order $I_x$, define the \emph{replacement of $X$ by $\{I_x: x \in X\}$} to be the order obtained by replacing every point $x \in X$ with $I_x$. We denote\footnote{The order $X(I_x)$ is sometimes called an \emph{ordered sum} and denoted $\sum_{x\in X} I_x$. However, in this paper we deal with cases in which $X(I_x)$ behaves more like a product-by-$X$ than a sum-over-$X$, so this notation is more convenient.} this order by $X(I_x)$. The underlying set of points in $X(I_x)$ is $\{(x, i): x\in X, i\in I_x\}$.  We will also refer to an order of the form $X(I_x)$ simply as a \emph{replacement}. We allow that for a given $x_0 \in X$ we have $I_{x_0} = \emptyset$, and in forming $X(I_x)$ think of replacing $x_0$ with a gap. The notion of a replacement generalizes the notion of a product, since if $I_x = Y$ for every $x$, then $X(I_x) = XY$. 

It will be useful to see in what sense the replacement operation is ``associative" on the left in its relation to the lexicographical product. Given a replacement $X(I_x)$ and some order $A$, we can multiply to form the product $A \times X(I_x)$. Points in this order are tuples of the form $(a, (x, i))$. Since this order is ``$A$-many copies of $X(I_x)$," each $I_x$ appears in it $A$-many times: every interval consisting of points of the form $(a, (x, \cdot))$ will be a copy of $I_x$, regardless of $a$. Alternatively, we might have begun with $A$, formed $A \times X$, and then replaced each point $(a, x)$ with $I_x$ to form $(A \times X)(I_x)$. Points in this order are of the form $((a, x), i)$. Every interval of points $((a, x), \cdot)$ is isomorphic to $I_x$ regardless of $a$, and hence this order is naturally isomorphic to $A \times X(I_x)$ under the map $((a, x), i) \mapsto (a, (x, i))$. 

However, to be consistent with our previous notation, the subscripts in the replacement $(A \times X)(I_x)$ should not (as they are written) range over $X$, but rather over $A \times X$. In forming the order the second way, by replacing after taking the product, we should have labeled the order replacing the point $(a, x)$ as $J_{(a, x)}$ and used the notation $(A \times X)(J_{(a, x)})$. If, as in our example, $J_{(a, x)} = I_x$ for every $a$, then $(A \times X)(J_{(a, x)})$ is isomorphic to $A \times X(I_x)$, as noted. 

We will use the following convention. If we first form the replacement $X(I_x)$ and then multiply on the left by $A$, we will write $AX(I_x)$. If we first form the product $AX$ and then replace each point $(a, x)$ with some $J_{(a, x)}$, we will use the notation $AX(J_{(a, x)})$.  It will sometimes be convenient to think of $AX(I_x)$ as being formed in the second way, with the left multiplication taking place first, in which case we will switch the notation to $AX(J_{(a, x)})$ and make clear that $J_{(a, x)}=I_x$ for all $a \in A$. 

Finally, let us note how isomorphisms factor through products and replacements. If $X$ and $Y$ are isomorphic orders, as witnessed by an isomorphism $f: X \rightarrow Y$, then for any order $Z$, the products $XZ$ and $YZ$ are isomorphic, as witnessed by the map $(x, z) \mapsto (f(x), z)$. Similarly, if we have collections of orders $I_x$, $x \in X$, and $J_y$, $y \in Y$ so that $I_x = J_{f(x)}$ for all $x \in X$, we will also have that the replacements $X(I_x)$ and $Y(J_y)$ are isomorphic, as witnessed by the map $(x, i) \mapsto (f(x), i)$. 

\subsection{Examples of countable orders isomorphic to their squares}

To begin, we give some examples of countable orders that are isomorphic to their own squares. In Section 5 we will construct uncountable examples. 

Our examples here rely on Cantor's theorem characterizing the order types of countable dense linear orders, as well as a generalization due to Skolem. 

\theoremstyle{definition}
\newtheorem*{Cantor}{Theorem}
\begin{Cantor}
(Cantor) If $X$ and $Y$ are countable dense linear orders without endpoints, then $X$ is isomorphic to $Y$. 
\end{Cantor}

Thus every countable dense linear order without endpoints is isomorphic to the set of rationals $\mathbb{Q}$ in their usual order. 

\theoremstyle{definition}
\newtheorem*{Skolem}{Theorem}
\begin{Skolem}
(Skolem) Fix some $k$, $1 \leq k \leq \omega$. Let $X, Y$ be countable dense linear orders without endpoints. Fix a partition $X = \bigcup_{i<k} X_i$ such that each $X_i$ is dense in $X$, and similarly $Y = \bigcup_{i<k} Y_i$. There is an isomorphism $f: X \rightarrow Y$ such that $f[X_i] = Y_i$ for every $i < k$. 
\end{Skolem}

This says that if we have two copies of the rationals, and color each of them with the same $k$ colors, using every color densely often, then there is an isomorphism between the two orders that respects the colorings. The proof is an easy generalization of the usual back-and-forth proof of Cantor's theorem. 

Now, if $L$ is any countable order, then $L\mathbb{Q}$ is countable, dense, and without endpoints. Hence, $L\mathbb{Q} \cong \mathbb{Q}$. In particular, $\mathbb{Q}^2 \cong \mathbb{Q}$, yielding our first example of an order isomorphic to its square.  To get another, let $X = \mathbb{Q}2$. This order is not isomorphic to $\mathbb{Q}$, since it contains many intervals isomorphic to $2$, and thus is not dense. Yet if $L$ is any countable order, then $LX = L(\mathbb{Q}2) \cong (L\mathbb{Q})2 \cong \mathbb{Q}2 = X$. In particular $X^2 \cong X$. 

Using Skolem's theorem, we generalize these examples. Fix $k$, $1 \leq k \leq \omega$, and a partition $\mathbb{Q} = \bigcup_{i<k} \mathbb{Q}_i$ such that each $\mathbb{Q}_i$ is dense in $\mathbb{Q}$. For each $i$, fix some countable order $I_i$, and form the replacement $X = \mathbb{Q}(I_q)$, where if $q \in \mathbb{Q}_i$ then $I_q = I_i$. This is the order obtained by replacing each point in the rationals with one of the $k$ countable orders $I_i$, such that each order appears densely often. 

Let $L$ be any countable order. Then $L\mathbb{Q}(I_q) \cong \mathbb{Q}(I_q)$, i.e. $LX \cong X$. The isomorphism follows from Skolem's theorem: $L\mathbb{Q}(I_q)$ is also a countable dense shuffling of the $I_i$, which is the same form as $\mathbb{Q}(I_q)$. 

To see this explicitly, let us write $L\mathbb{Q}(I_q)$ as $L\mathbb{Q}(J_{(l, q)})$, where $J_{(l, q)} = I_q$ for all $l \in L$. Thus $J_{(l, q)} = I_i$ if $q \in \mathbb{Q}_i$. We partition $L\mathbb{Q}$ according to the partition of $\mathbb{Q}$: let $Q_i = \{(l, q) \in L\mathbb{Q}: q \in \mathbb{Q}_i\}$. Then the $Q_i$ partition $L\mathbb{Q}$ and we have $J_{(l, q)} = I_i$ if $(l, q) \in Q_i$. Since each $Q_i$ is dense in $L\mathbb{Q}$, there is an isomorphism $f: L\mathbb{Q} \rightarrow \mathbb{Q}$ such that $f[Q_i] = \mathbb{Q}_i$ for every $i$. But then $J_{(l,q)} = I_{f((l, q))}$ for every $(l, q) \in L\mathbb{Q}$. Thus the isomorphism lifts to give an isomorphism of $L\mathbb{Q}(J_{(l, q)})$ with $\mathbb{Q}(I_q)$, i.e. of $LX$ with $X$. Since $L$ was arbitrary, we have in particular that $X^2 \cong X$. 

In Section 3 it will be shown that if $X$ is a countable order without endpoints and $X^n \cong X$ for some $n > 1$, then $X$ has the same form as the order above: $X \cong \mathbb{Q}(I_q)$, where for each $q$ there is a dense set of $p \in \mathbb{Q}$ such that $I_q = I_p$. It follows, as above, that $X^2 \cong X$ and hence $X^m \cong X$ for any $m > 1$. Of course, our main theorem is that $X^n \cong X \implies X^2 \cong X$ holds in general, but in the countable, no-endpoints case we have a complete classification: $X^n \cong X$ for some $n > 1$ if and only if $X \cong \mathbb{Q}(I_q)$. As we observed, such an order is not only invariant under left multiplication by itself, but by any countable order $L$. 

\subsection{A solution to the cube problem when $X$ has both endpoints}

It turns out that it is easy to prove $X^n \cong X \implies X^2 \cong X$ for linear orders $X$ with both a left and right endpoint. The proof uses the following theorem of Lindenbaum, that could be called the Schroeder-Bernstein theorem for linear orders. Although the proof for the general case is substantially harder, a generalized version of Lindenbaum's theorem does play an important role. 

\theoremstyle{definition}
\newtheorem*{Lindenbaum}{Theorem}
\begin{Lindenbaum}
(Lindenbaum) Suppose $X, Y$ are linear orders. If $X$ is isomorphic to an initial segment of $Y$, and $Y$ is isomorphic to a final segment of $X$, then $X \cong Y$. 
\end{Lindenbaum}

\begin{proof}
Suppose $f: X \rightarrow Y$ is an isomorphism of $X$ onto an initial segment of $Y$ and $g: Y \rightarrow X$ is an isomorphism of $Y$ onto a final segment of $X$. Then $f, g$ are in particular injections. Let $h: X \rightarrow Y$ be the bijection constructed out of $f, g$ as in the classical proof of the Schroeder-Bernstein theorem. Then the hypotheses guarantee that this bijection is order-preserving. 
\end{proof}

\theoremstyle{definition}
\newtheorem{Endpoints}[theorem]{Corollary}
\begin{Endpoints}\label{endpoints}
If $X$ is a linear order with both a left endpoint and right endpoint, and $X^n \cong X$ for some $n > 1$, then $X^2 \cong X$. 
\end{Endpoints}

\begin{proof}
Denote the minimal element of $X$ by $0$ and the maximal element by $1$. Then $X^2$ contains an initial segment isomorphic to $X$, namely the interval of points of the form $(0, \cdot)$. Since $X^n \cong X$, it follows that $X^2$ contains an initial segment isomorphic to $X^n$. But the interval of consisting of points $(1, 1, \ldots, 1, \cdot, \cdot)$ with $n-2$ leading 1s is a final segment of $X^n$ isomorphic to $X^2$. By Lindenbaum's theorem, $X^2 \cong X^n$. Hence $X^2 \cong X$. 
\end{proof}

Thus any linear order with both endpoints that is isomorphic to its cube is isomorphic to its square, solving the cube problem in this case. The simplicity of the proof suggests that there may be a simple proof when $X$ has either no endpoints or one endpoint. But the proof uses crucially that $X$ has both endpoints, and there does not seem to be an adaptation to the other cases. The isomorphisms constructed in Section 4 to take care of the other cases do make use of Schroeder-Bernstein style maps, but using them requires an understanding of the structure of orders $X$ that satisfy isomorphisms of the form $A^nX \cong X$.

\section{Characterizing orders $X$ such that $A^nX \cong X$}

\subsection{The isomorphism $AX \cong X$}

Throughout the rest of the paper, when it is not otherwise specified, $A$ simply refers to some fixed order. Our first task is to understand which orders $X$ are isomorphic to $AX$. We begin with some examples. 

\theoremstyle{definition}
\newtheorem{exfull}[theorem]{Example}
\begin{exfull}\label{exfull}
Let $X = A^{\omega}$. The map from $AX$ to $X$ defined by $(a, (a_0, a_1, \ldots)) \mapsto (a, a_0, a_1, \ldots)$ is order-preserving and bijective. Thus $AX \cong X$. 
\end{exfull}

Since it will appear often, let us denote the ``flattening map" in the example by $fl: A \times A^{\omega} \rightarrow A^{\omega}$, where $fl(a, (a_0, a_1, \ldots)) = (a, a_0, a_1, \ldots)$. 

\theoremstyle{definition}
\newtheorem{exsmall}[theorem]{Example}
\begin{exsmall}\label{exsmall}
Choose a distinguished element in $A$ and label it $0$. Let $X$ be the suborder of $A^{\omega}$ consisting of sequences that are eventually $0$. Then $AX \cong X$. Again the isomorphism is given by $fl$, restricted here to $AX$. Similarly, if $X$ is the subset of $A^{\omega}$ consisting of eventually constant sequences, then $AX$ is isomorphic to $X$ under the same map. 
\end{exsmall}

These examples can be generalized. The relevant notion is that of tail-equivalence of sequences in $A^{\omega}$. In what follows, $A^{<\omega}$ denotes the set of finite sequences of elements of $A$, including the empty sequence. We do not put any order structure on this set. If $r$ is a finite sequence and $u$ is a finite or infinite sequence, then $ru$ denotes the sequence obtained by concatenating $r$ with $u$. 

\theoremstyle{definition}
\newtheorem{tailequiv}[theorem]{Definition}
\begin{tailequiv}\label{tailequiv}
Two sequences $u, v \in A^{\omega}$ are called \emph{tail-equivalent} if there exist sequences $r, s \in A^{<\omega}$ (not necessarily of the same length) and a sequence $u' \in A^{\omega}$ such that $u = ru'$ and $v = su'$. If $u$ and $v$ are tail-equivalent, we write $u \sim v$. 
\end{tailequiv}

In dealing with sequences of elements of $A$, the letters $a, b, \ldots$ will usually denote single elements of $A$, whereas $r, s, \ldots$ will denote finite sequences, and $u, v, \ldots$ will denote infinite sequences. If $u = ru'$ for some finite sequence $r$, then $u'$ is called a \emph{tail-sequence} of $u$ and $r$ is called an \emph{initial sequence} of $u$. In an abuse, no distinction is made between elements of $A$ and sequences of length 1, so that $au$ refers to the sequence with first entry $a \in A$ followed by the tail-sequence $u \in A^{\omega}$. The length of a finite sequence $r$ is denoted $|r|$. Given $u \in A^{\omega}$, $u_i$ refers to the $i$th entry of $u$, and $u \upharpoonright n = (u_0, u_1, \ldots, u_{n-1})$ denotes the initial sequence of the first $n$ entries of $u$. If $u \sim v$, so that $u = ru'$ and $v = su'$ for some $r, s, u'$, the pair $ru'$ and $su'$ is called a \emph{meeting representation} of $u$ and $v$. 

Meeting representations are not unique: if $u=ru'$, $v=su'$, and $a$ is the first entry of $u'$, so that $u'=au''$, then letting $r' = ra$, $s'=sa$ we have that $u=r'u''$ and $v=s'u''$, giving a different representation. Usually, unzipping along a tail-sequence like this is the only way of generating distinct representations. However, if $u'$ is an eventually periodic sequence there are other ways. We will return to this issue later. 

It is easy to see that tail-equivalence is an equivalence relation on $A^{\omega}$. The equivalence class of $u$ is denoted $[u]$. It consists exactly of those elements in $A^{\omega}$ of the form $ru'$, where $r \in A^{<\omega}$ and $u'$ is a tail-sequence of $u$. 

The tail-equivalence classes are the smallest subsets of $A^{\omega}$ that are ``invariant under left multiplication by $A$" in the following sense. If $[u]$ is a tail-equivalence class, then viewing it as a suborder of $A^{\omega}$, we may form the product $A[u] = \{(a, v): a \in A, v \in [u]\}$. Then $fl[A[u]] = [u]$, that is, the flattening map $(a, v) \mapsto av$ witnesses the isomorphism $A[u] \cong [u]$. To check this, note that if $v \in [u]$, then $av \sim v$ and hence $av \in [u]$ for any $a \in A$. Hence $fl[A[u]] \subseteq [u]$. And if $v \in [u]$, say $v = (v_0, v_1, \ldots)$, then $v' = (v_1, v_2, \ldots)$ is tail-equivalent to $v$ and hence also in $[u]$. But then $v_0v' = v$ is in the image of $A[u]$, giving the reverse containment. 

On the other hand, if $X \subseteq A^{\omega}$ and $fl[AX]=X$, then $X$ is a union of tail-equivalence classes. For, if $u = (u_0, u_1, \ldots)$ is in $X$, then $fl^{-1}(u)=(u_0, (u_1, \ldots))$ is in $AX$ and therefore the tail-sequence $(u_1, u_2, \ldots)$ is in $X$. By induction, any tail-sequence $u'$ of $u$ is in $X$. Then for any $a \in A$, we must have $(a, u') \in AX$ and hence $fl((a, u')) = au' \in X$. By induction, for any finite sequence $r$ we have $ru' \in X$. Thus $[u] \subseteq X$, giving the claim. We have proved the following proposition. 

\theoremstyle{definition}
\newtheorem{proptailequiv}[theorem]{Proposition}
\begin{proptailequiv}\label{proptailequiv}
Fix a suborder $X \subseteq A^{\omega}$. Then $fl[AX] = X$ if and only if $X$ is a union of tail-equivalence classes. 
\end{proptailequiv}

This fact can be formulated dynamically. We may view $A^{<\omega}$ as the free semigroup on the alphabet $A$, where the semigroup operation is concatenation. This semigroup naturally acts on $A^{\omega}$ by concatenation on the left. The tail-equivalence classes are the smallest subsets of $A^{\omega}$ that are closed under this action by $A^{<\omega}$. 

Our aim now is to find the most general form of an order $X$ such that $AX \cong X$. Proposition \ref{proptailequiv} gives immediately that if $X$ is any suborder of $A^{\omega}$ that is a union of tail-equivalence classes, then $AX$ is isomorphic to $X$ under the flattening map $(a, u) \mapsto au$. Notice that the $X$'s appearing in Examples \ref{exfull} and \ref{exsmall} are of this form: $A^{\omega}$ is the union of all of its tail-equivalence classes, and the set of eventually $0$ sequences is exactly the single tail-equivalence class of the constant sequence $(0, 0, \ldots)$. 

The proposition actually yields substantially more general examples. Suppose that for every $u \in A^{\omega}$ we fix an order $I_u$, with the restriction that if $u \sim v$ then $I_u = I_v$. Let $X = A^{\omega}(I_u)$. Then $AX \cong X$, and the isomorphism is witnessed by the flattening isomorphism on the left two coordinates. To see this, rewrite $AX = A \times A^{\omega}(I_u)$ as $(A \times A^{\omega})(J_{(a, u)})$, where $J_{(a, u)}=I_u$ for all $a \in A$. Then $J_{(a, u)} = I_{au}$ as well, since $au \sim u$. This is the same as saying that the interval consisting of points $(a, u, \cdot)$ in $(A \times A^{\omega})(J_{(a, u)})$ is isomorphic to the interval of points $(au, \cdot)$ in $A^{\omega}(I_u)$. Thus the map $(a, u, x) \mapsto (au, x)$ makes sense, and defines an order isomorphism of $(A \times A^{\omega})(J_{(a, u)})$ with $A^{\omega}(I_u)$, that is, of $AX$ with $X$.  

Letting $I_{[u]}$ denote the single order $I_v$ for all $v \in [u]$, we may denote $X$ by $A^{\omega}(I_{[u]})$. Our previous examples were actually of this form: if $X \subseteq A^{\omega}$ is a union of tail-equivalence classes then $X$ may be written as $A^{\omega}(I_{[u]})$, where $I_{[u]} = 1$ if $u \in X$, and otherwise $I_{[u]}=\emptyset$. Of course, since the $I_{[u]}$ may be chosen arbitrarily, there are many other examples besides these. 

Here is a concrete one. Let $\mathbb{Z}$ denote the integers in their usual order, and form the product $\mathbb{Z}^{\omega}$ (this order is isomorphic to the irrationals). Each tail-equivalence class $[u] \subseteq \mathbb{Z}^{\omega}$ is a countable dense subset of $\mathbb{Z}^{\omega}$, hence the number of classes is $2^{\omega}$. Enumerate them as $\{C_{\alpha}: \alpha < 2^{\omega}\}$. Let $X$ be the order obtained by replacing every point in the $\alpha$th class with the ordinal $\alpha$, i.e. if $[u]=C_{\alpha}$ let $I_{[u]}=\alpha$, and let $X = \mathbb{Z}^{\omega}(I_{[u]})$. This order may be visualized as densely many copies of each ordinal less than $2^{\omega}$ interspersed with one another. By our observations above, $\mathbb{Z}X \cong X$. 

We shall now prove that this form is general: any order $X$ such that $AX \cong X$ must be of the form $A^{\omega}(I_{[u]})$. 

\theoremstyle{definition}
\newtheorem{thmax}[theorem]{Theorem}
\begin{thmax}\label{thmax}
Fix an order $X$. Then $AX \cong X$ if and only if $X$ is isomorphic to an order of the form $A^{\omega}(I_{[u]})$. 
\end{thmax}

\begin{proof}
We have already seen that if $X$ is of the form $A^{\omega}(I_{[u]})$ then $X$ is isomorphic to $AX$. 

So assume that $AX \cong X$, as witnessed by an isomorphism $f: AX \rightarrow X$. For every $a \in A$, let $aX$ denote the set of pairs in $AX$ with first coordinate $a$. Let $X_a$ denote $f[aX]$. Each $aX$ is an interval in $AX$ isomorphic to $X$, and hence each $X_a$ is an interval in $X$ isomorphic to $X$. These intervals are disjoint and cover $X$, and if $a < b$, then the interval $X_a$ lies to the left of $X_b$, that is, $X_a < X_b$. We think of $AX \cong X$ as saying that $X$ can be split into $A$-many copies of $X$, and $X_a$ as being the $a$th copy of $X$ within $X$. Let $f_a: X \rightarrow X_a$ be the isomorphism defined by $f_a(x) = f(a, x)$. 

Given one of the intervals $X_b$ and some $a \in A$, we may think of the image $f_a[X_b]$ as the $b$th copy of $X$ within $X_a$. Denote this interval by $X_{(a, b)}$. Then $f_a \circ f_b$ is an isomorphism of $X$ with $X_{(a,b)}$. Extending this, given a finite sequence $r = (a_0, a_1, \ldots, a_n) \in A^{<\omega}$, define $X_r$ as $f_{a_0} \circ f_{a_1} \circ \ldots \circ f_{a_n}[X]$. Denote the isomorphism $f_{a_0} \circ f_{a_1} \circ \ldots \circ f_{a_n}: X \rightarrow X_{r}$ by $f_{r}$. 

It is immediate that for any $r, s \in A^{<\omega}$, we have $f_r[X_s] = X_{rs}$, and conversely, $f_r^{-1}[X_{rs}] = X_s$. If $t$ extends $r$, then $X_t$ is a subinterval of $X_r$. If neither one of $r, t$ extends the other, and they differ first in the $n$th place with $r_n < t_n$, then $X_r < X_t$ in $X$.

Now, given an infinite sequence $u \in A^{\omega}$, define $I_u$ to be the interval obtained by taking the natural nested intersection:
\[
I_u = \bigcap_{n \in \omega} X_{u \upharpoonright n}.
\]

For a given $u$, the interval $I_u$ need not be isomorphic to $X$. Indeed, $I_u$ may be empty. However, every $x \in X$ is in some $I_u$. Further if $u < v$ in $A^{\omega}$, then for some $n$ we have $u \upharpoonright n < v \upharpoonright n$. Hence $X_{u \upharpoonright n} < X_{v \upharpoonright n}$ in $X$, and so $I_u < I_v$ in $X$. Thus we have that the $I_u$ in fact partition $X$, and further that they respect the lexicographical ordering of their indices. This shows $X \cong A^{\omega}(I_u)$. 

It remains to show that for $u \sim v$, we have $I_u \cong I_v$. If $u=ru'$ and $v=su'$, then if we begin in $X_r$ and $X_s$ (each copies of $X$) and descend in each interval along the same nested sequence of subintervals prescribed by $u'$, we arrive at $I_u$ and $I_v$, respectively. Hence, $I_u \cong I_v$. The isomorphism is given by $f_s \circ f_r^{-1}$, restricted to $I_u$. To see this explicitly, note that we have
\[
I_u = \bigcap_{n} X_{r(u' \upharpoonright n)},
\]
where $r(u\upharpoonright n)$ means $r$ concatenated with $(u \upharpoonright n)$. 
Thus, since $f_s \circ f_r^{-1}: X_r \rightarrow X_s$ is an isomorphism, we have

\[
 \begin{array}{ccl}
f_s \circ f_r^{-1}[I_u] & = & f_s \circ f_r^{-1}[\bigcap_{n} X_{r(u' \upharpoonright n)}]\\
& = & \bigcap_{n} f_s \circ f_r^{-1} [X_{r(u' \upharpoonright n})] \\
& = & \bigcap_{n} X_{s(u' \upharpoonright n)} \\
& = & I_{su'} \\ 
& = & I_v
\end{array}
\]

That is, $f_s \circ f_r^{-1}$ witnesses $I_u \cong I_v$, as claimed.

Since $u \sim v$ implies $I_u \cong I_v$, we may write $I_{[u]}$ for the single order type of $I_v$ for all $v \in [u]$. We have showed $X \cong A^{\omega}(I_{[u]})$.
\end{proof}

This theorem plays a basic role in the remainder of the paper, and some discussion is warranted. We have seen that if $X = A^{\omega}(I_{[u]})$ for some collection of orders $I_{[u]}$, then there is a natural isomorphism $f: AX \rightarrow X$ defined by $(a, u, x) \mapsto (au, x)$. That is, $f$ is the flattening map on the first two coordinates and the identity on the last. We write $f = (fl, \textrm{id})$. 

Even if we do not have identity, but only an isomorphism $F: X \rightarrow A^{\omega}(I_{[u]})$, there is still a natural isomorphism $f: AX \rightarrow X$, namely $f = F^{-1} \circ (fl, \textrm{id}) \circ (\textrm{id}, F)$. This is the same flattening map, up to the relabeling $F$. 

On the other hand, suppose we are given an isomorphism $f: AX \rightarrow X$. The proof of Theorem \ref{thmax} shows that we can use this $f$ to build an isomorphism $F$ from $X$ onto an order of the form $A^{\omega}(I_{[u]})$. In fact the proof shows more: if we view $f$, via the relabeling $F$, as an isomorphism of $A \times A^{\omega}(I_{[u]})$ with $A^{\omega}(I_{[u]})$, then $f$ is just the flattening map on the first two coordinates. (It need not be true that $f$ is the identity on the last coordinate, but we shall ignore this subtlety.) 

In this sense, there is a correspondence between isomorphisms $f: AX \rightarrow X$ and labelings $F: X \rightarrow A^{\omega}(I_{[u]})$. Given a labeling of $X$ as $A^{\omega}(I_{[u]})$, we have the natural flattening isomorphism between $AX$ and $X$, and given an isomorphism $f: AX \rightarrow X$, we can relabel $X$ as $A^{\omega}(I_{[u]})$ in such a way that $f$ is just the flattening isomorphism. 

\subsection{The isomorphism $A^2X \cong X$}
Our next goal is to characterize, for a fixed $n > 1$, those $X$ for which $A^nX \cong X$. We focus here on the case when $n=2$, but the work easily generalizes. A list of the corresponding definitions and results for the $n$ case is provided later in the section. 

In a sense, we have already succeeded in getting a characterization. By Theorem \ref{thmax}, $A^2X \cong X$ if and only if $X$ is isomorphic to an order of the form $(A^2)^{\omega}(I_{[u]}$), where now $[u]$ denotes a tail-equivalence class in $(A^2)^{\omega}$. But $(A^2)^{\omega}$ is naturally isomorphic to $A^{\omega}$ under the map $((a_0, a_1), (a_2, a_3), \ldots) \mapsto (a_0, a_1, a_2, a_3, \ldots)$. It is useful to characterize such an $X$ as a replacement of $A^{\omega}$ instead of as a replacement of $(A^2)^{\omega}$. 

\theoremstyle{definition}
\newtheorem{2tailequiv}[theorem]{Definition}
\begin{2tailequiv}\label{2tailequiv}
Two sequences $u, v \in A^{\omega}$ are called \emph{2-tail-equivalent} if there exist $r, s \in A^{<\omega}$ with $|r| \equiv |s| \pmod{2}$ and a sequence $u' \in A^{\omega}$ such that $u = ru'$ and $v = su'$. If $u$ and $v$ are 2-tail-equivalent, we write $u \sim_2 v$. 
\end{2tailequiv}

The definition of 2-tail-equivalence is identical to that of tail-equivalence, except that in the meeting representation $u=ru'$, $v=su'$ witnessing $u \sim_2 v$, the lengths of $r$ and $s$ are required to have the same parity. Reflexivity and symmetry of the relation $\sim_2$ are immediate, and transitivity can be checked. Hence 2-tail-equivalence is an equivalence relation. The $\sim_2$-equivalence class of $u$ is denoted $[u]_2$.  

We will need to see precisely how tail-equivalence classes are related to 2-tail-equivalence classes. If $u \sim_2 v$, then certainly $u \sim v$ as well. Hence the 2-tail-equivalence relation refines the tail-equivalence relation, and every $\sim$-class $[u]$ splits into some number of $\sim_2$-classes. Clearly in fact, it splits into at most two such classes: if $v \sim u$ as witnessed by $u = ru'$ and $v = su'$, then either $|s| \equiv |r| \pmod{2}$ or $|s| \equiv |r| + 1 \pmod{2}$. In the first case we have $v \sim_2 u$, and in the second, for any $a \in A$, we have $|s| \equiv |ar| \pmod{2}$, and hence $v \sim_2 au$. Thus $[u] = [u]_2 \cup \, [au]_2$ for any fixed $a \in A$. 

The classes $[u]_2$ and $[au]_2$ are either equal or disjoint, depending on whether $au \sim_2 u$ or not. If $au \sim_2 u$, we have $[u] = [u]_2$; otherwise $[u]$ splits into the two disjoint classes $[u]_2$ and $[au]_2$. 

Usually, $au \not \sim_2 u$. Certainly the obvious representation $u = ru'$, $au = su'$, where $r = \emptyset$, $s = a$, and $u' = u$ does not witness $au \sim_2 u$. Most of the time this obvious representation is the only one, up to unzipping along the tail-sequence. However, if $u$ is an eventually periodic sequence then it is possible to get truly distinct representations, and in some of these cases we have $au \sim_2 u$.  

For any finite sequence $s \in A^{<\omega}$, let $\overline{s}$ denote the sequence $sss\ldots \in A^{\omega}$. Define a sequence $u \in A^{\omega}$ to be \emph{eventually periodic} if there exist $r, s \in A^{<\omega}$ such that $u = rsss\ldots = r\overline{s}$. Such a sequence is said to be repeating in $s$. If $r=\emptyset$ we say $u$ is simply \emph{periodic}. If $n$ is the shortest possible length of a sequence in which $u$ repeats, we say $u$ has period $n$.

To prove the following proposition, we use the fact that if $n$ is the period of $u$ and we have a ``periodic representation" of $u$ as $u=r\overline{s}$, then $n$ divides $|s|$.

\theoremstyle{definition}
\newtheorem{period}[theorem]{Proposition}
\begin{period}\label{period}
Fix elements $u \in A^{\omega}$ and $a \in A$. Then $au \sim_2 u$ if and only if $u$ is eventually periodic and the period of $u$ is odd. 
\end{period}
\begin{proof}
Suppose $u$ is repeating in $s$, and $|s|$ is of minimal length, i.e. $|s|$ = period of $u$. Assume first that $|s| \equiv 1 \pmod 2$. We have $u = r\overline{s}$ for some $r \in A^{<\omega}$. Let $u' = \overline{s}$ be the periodic tail-sequence, and let $t=ars$. Then since $|s|$ is odd, we have $|t| = |ars| \equiv |r| \pmod{2}$. Thus the meeting representation $u=ru'$ and $au=tu'$ witnesses $au \sim_2 u$. 

Conversely, suppose $au \sim_2 u$, as witnessed by $u = ru'$ and $au = (ar')u'$, with $|r| \equiv |ar'| \pmod{2}$. Hence $|r| \not\equiv |r'| \pmod{2}$. Since both $r$ and $r'$ are initial sequences of $u$, it must be that one is an initial sequence of the other. Assume $r'$ is an initial sequence of $r$, that is, $r = r's$ for some $s \in A^{\omega}$. Then since $r, r'$ are inequivalent modulo 2, it must be that $|s|$ is odd. Now, on one hand $u = ru' = r'su'$, and on the other $u = r'u'$. Hence $u' = su'$. The only way this is possible is if $u' = \overline{s}$. Therefore $u=r\overline{s}$ is an eventually periodic sequence. It may be that the period of $u$ is shorter than $|s|$, but since $|s|$ is odd, the period of $u$ must be odd as well. The case when $r$ is an initial sequence of $r'$ is similar. 
\end{proof}

Thus $[u] = [u]_2$ if and only if $u$ is eventually periodic and of odd period. Since they play an especially important role in what follows, we emphasize that for sequences $u$ of period 1 we have $[u] = [u]_2 = [au]_2$. For sequences $u$ that are not eventually periodic we have $[u]_2 \cap [au]_2 = \emptyset$, where $a$ is any fixed element of $A$. 

Just as the $\sim$-classes $[u]$ were the smallest suborders of $A^{\omega}$ invariant under left multiplication by $A$, the $\sim_2$-classes $[u]_2$ are the smallest suborders of $A^{\omega}$ invariant under left multiplication by $A^2$. Let $fl_2: A^2 \times A^{\omega} \rightarrow A^{\omega}$ denote the ``2-flattening isomorphism" defined by $(a, b, u) \mapsto abu$. 

\theoremstyle{definition}
\newtheorem{prop2tailequiv}[theorem]{Proposition}
\begin{prop2tailequiv}\label{prop2tailequiv}
Fix a suborder $X \subseteq A^{\omega}$.  Then $fl_2[A^2X] = X$ if and only if $X$ is a union of $\sim_2$-classes. 
\end{prop2tailequiv}
\begin{proof}
The proof is essentially identical to the proof of Proposition \ref{proptailequiv}. There, the relevant fact was that $au \sim u$ for all $a, u$. Here, it is that $abu \sim_2 u$ for all $a, b, u$. 
\end{proof}

Just as $fl$ sends $A[u]$ isomorphically onto $[u]$ for every $u \in A^{\omega}$, $fl_2$ sends $A^2[u]_2$ isomorphically onto $[u]_2$. Less formally, we can express this by saying $A[u] = [u]$ and $A^2[u]_2 = [u]_2$ for all $u \in A^{\omega}$. A natural question is what happens to a given 2-tail-equivalence class $[u]_2$ when it is multiplied by a single factor of $A$.

\theoremstyle{definition}
\newtheorem{Au2}[theorem]{Proposition}
\begin{Au2}\label{Au2}
For all $u \in A^{\omega}$ and $a \in A$ we have $fl[A[u]_2] = [au]_2$. 
\end{Au2}
\begin{proof}
If $x \in fl[A[u]_2]$, then $x = bv$ for some $b \in A$ and $v \sim_2 u$. But then $bv \sim_2 au$, i.e. $x \sim_2 au$. Conversely, if $x \sim_2 au$ then writing $x=x_0x'$ we have that $x' \sim_2 u$ and hence $x \in fl[A[u]_2]$. 
\end{proof}

Informally this says $A[u]_2 = [au]_2$. It follows that $A[au]_2 = [u]_2$, from which we recover $A^2[u]_2 = [u]_2$. 

We can use Proposition \ref{prop2tailequiv} to produce orders $X$ such that $A^2X \cong X$. For every class $[u]_2$, fix an order $I_{[u]_2}$. Let $A^{\omega}(I_{[u]_2})$ denote the order $A^{\omega}(I_u)$, where if $v \sim_2 u$, then $I_v = I_u = I_{[u]_2}$. Then $A^2X \cong X$. The isomorphism is given by $((a, b), u, x) \mapsto (abu, x)$. This map makes sense because both the interval consisting of points $((a, b), u, \cdot)$ in $A^2X$ and the interval of points $(abu, \cdot)$ in $X$ are of type $I_{[u]_2} = I_{[abu]_2}$. 

The form $A^{\omega}(I_{[u]_2})$ is general for orders $X$ isomorphic to $A^2X$. 

\theoremstyle{definition}
\newtheorem{AsquaredX}[theorem]{Theorem}
\begin{AsquaredX}\label{AsquaredX}
Fix an order $X$. Then $A^2X \cong X$ if and only if $X$ is isomorphic to an order of the form $A^{\omega}(I_{[u]_2})$. 
\end{AsquaredX}
\begin{proof}
We have already showed the forward direction. Assume $A^2X \cong X$. Then by Theorem \ref{thmax} we have that $X$ is of the form $(A^2)^{\omega}(J_{[w]})$, where $[w]$ denotes a tail-equivalence class in $(A^2)^{\omega}$. Let $Fl: (A^2)^{\omega} \rightarrow A^{\omega}$ be the isomorphism defined by $((a_0, a_1), (a_2, a_3), \ldots) \mapsto (a_0, a_1, a_2, \ldots)$. 

It is straightforward to check that for $w, x \in (A^2)^{\omega}$, we have $w \sim x$ (in the sense of $(A^2)^{\omega}$) if and only if $Fl(w) \sim_2 Fl(x)$ (in the sense of $A^{\omega}$). 

For every $u \in A^{\omega}$, define $I_u$ to be $J_{Fl^{-1}(u)}$. Then $J_{w} = I_{Fl(w)}$ for every $w \in (A^2)^{\omega}$. Hence the isomorphism $Fl: (A^2)^{\omega} \rightarrow A^{\omega}$ lifts to give an isomorphism of $(A^2)^{\omega}(J_{w})$ with $A^{\omega}(I_u)$. Since $w \sim x$ in $(A^2)^{\omega}$ implies $J_w = J_x$, we have that $u \sim_2 v$ in $A^{\omega}$ implies $I_u = I_v$. Thus $X$ is of the form $A^{\omega}(I_{[u]_2})$. 
\end{proof}

\subsection{When $A^2X \cong X$ implies $AX \cong X$} 

We now turn to the question of when the isomorphism $A^2X \cong X$ implies $AX \cong X$. Notice that if $A^2X \cong X$ and we use Theorem \ref{AsquaredX} to decompose $X$ as $A^{\omega}(I_{[u]_2})$, then there is no ``obvious" isomorphism between $AX$ and $X$. The natural guess for such an isomorphism would be the flattening map $(a, u, x) \mapsto (au, x)$. But the definition of this map may be meaningless: the interval consisting of tuples $(a, u, \cdot)$ in $AX$ is of type $I_u = I_{[u]_2}$ whereas the interval $(au, \cdot)$ in $X$ is of type $I_{[au]_2}$. If $u \not\sim_2 au$, these orders may be distinct. 

If it so happens that $I_{[u]_2} = I_{[au]_2}$ for every $u$ and $a$, then the flattening map makes sense and witnesses $AX \cong X$. Indeed if this is so we may denote the common order type of $I_{[u]_2}$ and $I_{[au]_2}$ by $I_{[u]}$, so that $X \cong A^{\omega}(I_{[u]})$. 

But this need not be the case. For example, consider $\mathbb{Z}^{\omega}$ and let $v$ denote the sequence $(1, 2, 3, \ldots)$. Then $v$ is not eventually periodic and hence not 2-tail-equivalent to $0v = (0, 1, 2, \ldots)$. Let $I_{[v]_2} = 1$ and $I_{[0v]_2}=2$. For all other $\sim_2$-classes $[u]_2$, let $I_{[u]_2}= \emptyset$. Let $X = \mathbb{Z}^{\omega}(I_{[u]_2})$. Then the map $((a, b), u, x) \mapsto (abu, x)$ witnesses $\mathbb{Z}^2X \cong X$, while the ``map" $(a, u, x) \mapsto (au, x)$, which purports to send each copy of $1$ in $AX$ isomorphically onto a copy of $2$ in $X$, and vice versa, is meaningless. In this sense it is not immediate that $\mathbb{Z}X \cong X$ (though, in this case, this turns out to be true). 

The following proposition describes the precise relationship between $AX$ and $X$ when $A^2X \cong X$. It echoes Proposition \ref{Au2}. 

\theoremstyle{definition}
\newtheorem{switching}[theorem]{Proposition}
\begin{switching}\label{switching}
Suppose $X$ is an order such that $A^2X \cong X$, so that $X$ is of the form $A^{\omega}(I_{[u]_2})$, and let $Y = AX$. Then $Y \cong A^{\omega}(J_{[u]_2})$, where for all $u \in A^{\omega}$ and $a \in A$, we have $J_{[u]_2} = I_{[au]_2}$. 
\end{switching}
\begin{proof}
Define $J_{[u]_2}=I_{[au]_2}$ as in the statement of the theorem, and assume for simplicity that $X$ is not only isomorphic to but in fact equals $A^{\omega}(I_{[u]_2})$. Then the isomorphism between $AX$ and $A^{\omega}(J_{[u]_2})$ is given by the flattening map $(a, u, x) \mapsto (au, x)$. This is an isomorphism since each interval in $AX$ consisting of points $(a, u, \cdot)$ is of type $I_{[u]_2}$ and each interval $(au, \cdot)$ in $A^{\omega}(J_{[u]_2})$ is of type $J_{[au]_2} = I_{[aau]_2} = I_{[u]_2}$ as well. 
\end{proof}

Thus if $A^2X \cong X$, we obtain $AX$ by interchanging the role of each $I_{[u]_2}$ with $I_{[au]_2}$ in the decomposition $X \cong A^{\omega}(I_{[u]_2})$. It is worth noting that in the case when $[u]_2 = [au]_2 = [u]$, no interchange is needed: in this case the orders $I_{[u]_2}$, $I_{[au]_2}$, and $J_{[u]_2}$, $J_{[au]_2}$ are all identical, whereas in general only the equalities $J_{[u]_2}=I_{[au]_2}$ and $I_{[u]_2} = J_{[au]_2}$ hold. 

The upshot of Proposition \ref{switching} is that, in the case where $A^2X \cong X$, if we can find an order automorphism of $A^{\omega}$ that maps each $\sim_2$-class $[u]_2$ onto $[au]_2$, we can lift it to obtain an isomorphism of $AX$ with $X$. 

\theoremstyle{definition}
\newtheorem{pra}[theorem]{Definition}
\begin{pra}\label{pra}
An order automorphism $f: A^{\omega} \rightarrow A^{\omega}$ is called a \emph{parity-reversing automorphism} (abbreviated p.r.a.) if $f(u) \in [au]_2$ for every $u \in A^{\omega}$ and $a \in A$. Equivalently, an order automorphism $f$ of $A^{\omega}$ is a p.r.a. if for every $u \in A^{\omega}$ and $a \in A$ the image of $[u]_2$ under $f$ is $[au]_2$.
\end{pra}

It follows that the image of $[au]_2$ under a parity-reversing automorphism $f$ is $[u]_2$. Parity-reversing automorphisms are ``idempotent on $\sim_2$-classes" in this sense. 

\theoremstyle{definition}
\newtheorem{proppra}[theorem]{Proposition}
\begin{proppra}\label{proppra}
Suppose that $A^{\omega}$ admits a parity-reversing automorphism. Then for every order $X$, if $A^2X \cong X$ then $AX \cong X$.
\end{proppra}

\begin{proof}
Let $f: A^{\omega} \rightarrow A^{\omega}$ be a parity-reversing automorphism. Fix $X$ and assume $A^2X \cong X$. Writing $X$ as $A^{\omega}(I_{[u]_2})$ and $AX$ as $A^{\omega}(J_{[u]_2})$, we may define an isomorphism from $X$ to $AX$ by the map $(u, x) \mapsto (f(u), x)$. This map is well-defined: the interval $(u, \cdot)$ in $X$ is of type $I_{[u]_2}$, and the interval $(f(u), \cdot)$ in $AX$ is of type $J_{[au]_2} = I_{[u]_2}$ since $f(u) \in [au]_2$. Hence $X \cong AX$. 
\end{proof}

In the context of the cube problem, the case of interest is when $A=X$. If $X$ is an order such that $X^3 \cong X$, then we may rewrite this as $X^2X \cong X$. By \ref{proppra}, if $X^{\omega}$ has a p.r.a., then $XX \cong X$ as well, i.e. $X^2 \cong X$. In Section 5, parity-reversing isomorphisms for $A^{\omega}$ are constructed for many different kinds of orders $A$, culminating in the proof that if $X^3 \cong X$ then indeed $X^{\omega}$ has a p.r.a., no matter the cardinality of $X$.

\subsection{The isomorphism $A^nX \cong X$} We provide the analogous definitions and results (without proof) for analyzing orders that satisfy the isomorphism $A^nX \cong X$, for some fixed $n > 1$. 

\theoremstyle{definition}
\newtheorem{ntailequiv}[theorem]{Definition}
\begin{ntailequiv}\label{ntailequiv}
 Two sequences $u, v \in A^{\omega}$ are called \emph{$n$-tail-equivalent} if there exist finite sequences $r, s \in A^{<\omega}$ with $|r| \equiv |s| \pmod{n}$ and a sequence $u' \in A^{\omega}$ such that $u = ru'$ and $v = su'$. If $u$ and $v$ are $n$-tail-equivalent, we write $u \sim_n v$. 
\end{ntailequiv}

It may be checked that $\sim_n$ is an equivalence relation. The $\sim_n$-class of $u$ is denoted $[u]_n$. The characterization of the isomorphism $A^nX \cong X$ is given by the following theorem. 

\theoremstyle{definition}
\newtheorem{AnX}[theorem]{Theorem}
\begin{AnX}\label{AnX}
For an order $X$, we have $A^nX \cong X$ if and only if $X$ is isomorphic to an order of the form $A^{\omega}(I_{[u]_n})$. 
\end{AnX}

The next theorem describes the relationship between the various orders $A^kX$, $1 \leq k \leq n-1$, in the case where $A^nX \cong X$. The notation $a^k$ denotes the sequence $aa\ldots a \in A^{<\omega}$ in which $a$ is repeated $k$ times. 

\theoremstyle{definition}
\newtheorem{revolving}[theorem]{Proposition}
\begin{revolving}\label{revolving}
Suppose $X$ is an order such that $A^nX \cong X$, so that $X$ is of the form $A^{\omega}(I_{[u]_n})$. For a fixed $k$, $1 \leq k \leq n-1$, let $Y = A^kX$. Then $Y \cong A^{\omega}(J_{[u]_n})$, where for all $u \in A^{\omega}$ and $a \in A$, we have $J_{[u]_n} = I_{[a^ku]_n}$. 
\end{revolving}

In the above, the sequence $a^k$ may be replaced with any $k$-sequence $a_0a_1 \ldots a_{k-1}$ of elements of $A$. 

\theoremstyle{definition}
\newtheorem{ra}[theorem]{Definition}
\begin{ra}\label{ra}
An order automorphism $f: A^{\omega} \rightarrow A^{\omega}$ is called an $n$-\emph{revolving automorphism} (abbreviated $n$-r.a.) if for every $u \in A^{\omega}$ and $a \in A$ the image of $[u]_n$ under $f$ is $[au]_n$. 
\end{ra}

\theoremstyle{definition}
\newtheorem{propra}[theorem]{Proposition}
\begin{propra}\label{propra}
If $A^nX \cong X$ and there exists an $n$-revolving automorphism on $A^{\omega}$, then $AX \cong X$ as well (and hence $A^kX \cong X$ for all $k$).
\end{propra}

Hence if $X^{n+1} \cong X$ for some $n \geq 1$ and there is an $n$-revolving automorphism on $X^{\omega}$, we have $X^2 \cong X$ as well. 

\subsection{A solution to the cube problem when $X$ is countable, without endpoints}

As a corollary to Theorem \ref{thmax}, we prove the characterization that was claimed in Section 2 for countable orders $X$ without endpoints such that $X^n \cong X$ for some $n>1$. A consequence is that the cube property holds for such orders. The argument differs substantially from the argument for the general case, and if desired this section can be skipped. 

\theoremstyle{definition}
\newtheorem{cntblnoend}[theorem]{Corollary}
\begin{cntblnoend}\label{cntblnoend}
Let $X$ be a countable linear order without endpoints. If $X^n \cong X$ for some $n > 1$, then $X$ is isomorphic to an order of the form $\mathbb{Q}(I_q)$, where for each $q$, there are densely many $p \in \mathbb{Q}$ with $I_q = I_p$. Hence for any countable order $L$ we have $LX \cong X$, and in particular $X^2 \cong X$. 
\end{cntblnoend}

\begin{proof}
Let $A = X^{n-1}$. The hypothesis is that $AX \cong X$. By Theorem \ref{thmax}, $X$ is isomorphic to an order of the form $A^{\omega}(I_{[u]})$. Note that since $A$ is countable, each equivalence class $[u]$ is countable: each $v \in [u]$ is of the form $ru'$ for some tail-sequence $u'$ of $u$, and there are only countably many $r \in A^{< \omega}$, and countably many tail-sequences $u'$. 

Since $X$ is countable, it must be that for all but countably many of the equivalence classes $[u]$ we have $I_{[u]} = \emptyset$. Let $k$ be the number of classes for which $I_{[u]}$ is nonempty, so that $1 \leq k \leq \omega$. Enumerate these classes as $C_i, i<k$, and let $I_i$ denote $I_{[u]}$ if $[u]=C_i$. Let $C = \bigcup_{i<k} C_i$. Since $C$ is the union of those classes $[u]$ for which $I_{[u]}$ is nonempty, we have that $X = C(I_u)$ where $I_u = I_i$ if $[u] = C_i$. We use the notation $X = C(I_i)$. 

Split $\mathbb{Q}$ into $k$-many dense subsets $\mathbb{Q} = \bigcup_{i<k} \mathbb{Q}_i$ and form the order $\mathbb{Q}(I_q)$, where if $q \in \mathbb{Q}_i$, then $I_q = I_i$. Denote this order by $\mathbb{Q}(I_i)$. If we can show that $X = C(I_i)$ is isomorphic to $\mathbb{Q}(I_i)$, the proof will be complete. To do this, it is sufficient to show that there is an isomorphism of $C$ with $\mathbb{Q}$ that sends each $C_i$ onto $\mathbb{Q}_i$. By Skolem's theorem, it is enough to show that $C$ is countable, dense and without endpoints, and each $C_i$ is dense in $C$.

Certainly $C$ is countable, since it is a countable union of countable classes. For the rest, note that since $X$ is without endpoints, $A = X^{n-1}$ is also without endpoints. Fix $v < w$ in $C$ and fix one of the $C_i$. Pick a representative $u \in C_i$, so that $C_i = [u]$. Since $v < w$ we may find an $n$ so that $v \upharpoonright n < w \upharpoonright n$ lexicographically. Since $A$ is without endpoints, we may find $a \in A$ such that $a > v_n$, and $b, c \in A$ such that $b < v_0$ and $c > w_0$. Let $x = (v \upharpoonright n)au$, $y = bu$, and $z = cu$. Then $y < v < x < w < z$ in $A^{\omega}$, and clearly $x, y, z \in [u] = C_i$. But this proves the claims: $C_i$ is dense in $C$, but further $C$ is dense and without a top or bottom point. The proof is complete. 
\end{proof}

Although we are going to show that the implication $X^n \cong X \implies X^2 \cong X$ holds for any linear order $X$, the advantage of the corollary is that it yields a complete classification in the countable, no endpoints case. Such an $X$ is of the form $\mathbb{Q}(I_i)$. Hence to specify $X$ it is enough to specify the number of parts $k$ in the partition $\mathbb{Q} = \bigcup_{i<k} \mathbb{Q}_i$ and the countable orders $I_i$. 

Under what other hypotheses can we carry out a similar proof to classify the orders satisfying the isomorphism $X^n \cong X$? We can always use Theorem \ref{thmax} to decompose such an $X$ into an order of the form $A^{\omega}(I_{[u]})$, where $A=X^{n-1}$. But the argument from there can break down in a number of ways. 

When $X$ is countable and has a single endpoint, the same proof goes through and yields an analogous classification. In the countable, two endpoints case, the proof does not carry over. The reason is that, while there is a unique (up to isomorphism) countable dense order with two endpoints, namely $Q = \mathbb{Q} \cap [0, 1]$, this order does not enjoy the same invariance under multiplication that $\mathbb{Q}$ enjoys. There exist countable orders $L$ with both endpoints such that $LQ \not \cong Q$. In fact it is easy to see we have non-isomorphism whenever $L$ is not dense. Moreover, this difference has consequences. Unlikes for $\mathbb{Q}$, there are right products of $Q$ that are not isomorphic to their squares. For example, $X=Q2$ is not isomorphic to its square. It follows that the analogous classification in this case fails, though it may be that some more detailed classification is possible. 

In the uncountable case, the proof breaks down completely, and the results of Section 5 show that no similar classification is possible. In many models of set theory, there do not even exist analogies to the order $\mathbb{Q}$ on higher cardinals. Even when such orders do exist, nothing like ``$X^n \cong X$ if and only if $X \cong \mathbb{Q}(I_i)$" is true for uncountable $X$. 

For example, under the continuum hypothesis there exists the saturated order $\mathbb{Q}_1$ of size $\aleph_1$. This order has a characterization a la Cantor's characterization of $\mathbb{Q}$, and enjoys many analogous properties to those of $\mathbb{Q}$. For one, if $L$ is any order of size at most $ \aleph_1$, then $L\mathbb{Q}_1 \cong \mathbb{Q}_1$, and in particular $\mathbb{Q}_1^2 \cong \mathbb{Q}_1$. Furthermore, we have the result ``any order $X$ of the form $\mathbb{Q}_1(I_i)$ is invariant under left multiplication by any $L$ of size at most $\aleph_1$. In particular, if the orders $I_i$ are of size at most $\aleph_1$, then for any $n$ we have $X^n \cong X$." (Here, each $I_i$ replaces densely many points in $\mathbb{Q}_1$.) But the converse is false outright: if $X$ is of size $\aleph_1$ and $X^n \cong X$ it need not be true that $X \cong \mathbb{Q}_1(I_i)$. For example, it is possible to use the methods of Section 5 to produce, without extra set theoretic hypotheses, an $X$ of size $\aleph_1$ such that $X^2 \cong X$ but $\omega_1X \not \cong X$.

\section{Parity-reversing automorphisms on $A^{\omega}$}

In this section we show how to construct parity-reversing automorphisms of $A^{\omega}$ for orders $A$ satisfying certain structural requirements. A consequence of our work is that if $X^3 \cong X$, then $X^{\omega}$ has a parity-reversing automorphism. By Proposition \ref{proppra} it follows that $X^2 \cong X$ as well. At the end of the section, we sketch how to generalize the argument to get the implication $X^n \cong X \implies X^2 \cong X$ for any order $X$ and $n \geq 2$, finishing the proof of the main theorem. 

For a sequence $u = (u_0, u_1, u_2, \ldots) \in A^{\omega}$, let $\sigma u$ denote the once shifted tail-sequence $(u_1, u_2, \ldots)$ and $\sigma^n u$ the $n$-times shifted tail-sequence $(u_n, u_{n+1}, \ldots)$.

As a first step toward constructing a p.r.a. for $A^{\omega}$, we show that any closed interval in $A^{\omega}$ with periodic endpoints of odd period has a parity-reversing automorphism. It is then possible to write down conditions under which a collection of these partial maps can be stitched together to get a full p.r.a. on $A^{\omega}$.

A technical note: if $r, s \in A^{<\omega}$ are finite sequences, it can be shown that if $\overline{r} < \overline{s}$ in $A^{\omega}$, then it always holds that $\overline{r} < \overline{rs} < \overline{sr} < \overline{s}$, even when one of the sequences is an initial sequence of the other. We will use this fact repeatedly. However, in every case in the subsequent work where we actually consider specific sequences $\overline{r} < \overline{s}$, it will be apparent that we actually have $\overline{r} < \overline{rs} < \overline{sr} < \overline{s}$. 

Now, suppose that we have finite sequences $r, s \in A^{<\omega}$ of lengths $m, n$ respectively, such that $\overline{r} < \overline{s}$ in $A^{\omega}$. The shift map $u \mapsto \sigma^m u$ restricted to the interval $[\overline{r}, \overline{rs}]$ is an isomorphism of this interval with the interval $[\overline{r}, \overline{sr}]$. The inverse map is given by $u \mapsto ru$, restricted to $[\overline{r}, \overline{sr}]$. Similarly the map $u \mapsto su$, when restricted to $[\overline{rs}, \overline{s}]$, is an isomorphism of this interval with $[\overline{sr}, \overline{s}]$ whose inverse is given by $u \mapsto \sigma^n u$. Thus we have the following proposition. 

\theoremstyle{definition}
\newtheorem{lemstd}[theorem]{Lemma}
\begin{lemstd}\label{lemstd}
Suppose $r, s \in A^{<\omega}$ are finite sequences, respectively of lengths $m, n$, such that $\overline{r} < \overline{s}$ in $A^{\omega}$. Then the map $f: [\overline{r}, \overline{s}] \rightarrow [\overline{r}, \overline{s}]$ defined by 
\begin{equation*}
f(u) = \left\{ \begin{array}{rl}
\sigma^mu & u \in [\overline{r}, \overline{rs}] \\
su & u \in [\overline{rs}, \overline{s}]
\end{array} \right.
\end{equation*}
is an order automorphism of the interval $[\overline{r}, \overline{s}]$. Its inverse is given by 
\begin{equation*}
f^{-1}(u) = \left\{ \begin{array}{rl}
ru & u \in [\overline{r}, \overline{sr}] \\
\sigma^nu & u \in [\overline{sr}, \overline{s}].
\end{array} \right.
\end{equation*}
\end{lemstd}

Note that there is no ambiguity in the definition of $f(\overline{rs})$ in the statement of the lemma, since $\sigma^m\overline{rs} = s\overline{rs} = \overline{sr}$. 

Given an interval of the form $[\overline{r}, \overline{s}]$ in $A^{\omega}$, we call the automorphism $f$ defined in Lemma \ref{lemstd} a \emph{standard map} on $[\overline{r}, \overline{s}]$. Standard maps are the essential tool in our proof of the cube property for $(LO, \times)$.

Standard maps are defined with respect to given sequences $r, s \in A^{<\omega}$, and an interval can carry more than one standard map. For example, if $r_0 = a$, $r_1 = aa$, and $s=b$ for some points $a<b$ in $A$, then we have $[\overline{r_0}, \overline{s}]=[\overline{r_1}, \overline{s}]=[\overline{a}, \overline{b}]$. However, the standard map on this interval defined with respect to the sequences $r_0, s$ is different than the standard map defined with respect to the sequences $r_1, s$. We adopt the convention that the phrase ``the standard map on $[\overline{r}, \overline{s}]$" means the one defined with respect to the sequences $r, s$. 

\theoremstyle{definition}
\newtheorem{lemstdpra}[theorem]{Lemma}
\begin{lemstdpra}\label{lemstdpra}
Fix $r, s \in A^{<\omega}$ of lengths $m, n$ respectively. If $m, n$ are both odd, then the standard map $f: [\overline{r}, \overline{s}] \rightarrow [\overline{r}, \overline{s}]$ is parity-reversing, in the sense that for all $u \in [\overline{r}, \overline{s}]$ and $a \in A$, we have $f(u) \in [au]_2$. 
\end{lemstdpra}
\begin{proof}
Either $f(u) = \sigma^m u$ or $f(u) = su$. In either case, the obvious meeting representation between $au$ and $f(u)$ witnesses $f(u) \sim_2 au$. 
\end{proof}

\theoremstyle{definition}
\newtheorem{Abothpra}[theorem]{Theorem}
\begin{Abothpra}\label{Abothpra}
Suppose $A$ has both a left and right endpoint. Then $A^{\omega}$ has a parity-reversing automorphism. 
\end{Abothpra}
\begin{proof}
Let $0$ denote the left endpoint of $A$, and $1$ denote the right endpoint. Then $A^{\omega}$ also has left and right endpoints, namely $\overline{0}$ and $\overline{1}$. That is, $A^{\omega} = [\overline{0}, \overline{1}]$. Thus the standard map on this interval is actually an automorphism of $A^{\omega}$. It is parity-reversing by Lemma \ref{lemstdpra}, since $\overline{0}$ and $\overline{1}$ are both of period 1.
\end{proof}

Thus if $A$ has both a left and right endpoint, and $X$ is an order such that $A^2X \cong X$, then $AX \cong X$ as well. Decomposing  $X$ as $A^{\omega}(I_{[u]_2})$ and $AX$ as $A^{\omega}(J_{[u]_2})$, the isomorphism is given by $(u, x) \mapsto (f(u), x)$, where $f$ is the standard map on $A^{\omega} = [\overline{0}, \overline{1}]$. 

In this case, we may alternatively apply Lindenbaum's version of the Schroeder-Bernstein theorem to get an isomorphism between $X$ and $AX$: $X$, which is isomorphic to $A^2X$, contains an initial copy of $AX$ by virtue of the fact that $A$ has a left endpoint, and $AX$ contains a final copy of $X$ since $A$ has a right endpoint. Thus $AX \cong X$ by Lindenbaum's theorem. If $A = X$, we recover Corollary \ref{endpoints}.

These approaches are actually the same. The isomorphism one gets from the proof of Lindenbaum's theorem (which is really just the classical proof of the Schroeder-Bernstein theorem), when viewed as an isomorphism of $A^{\omega}(I_{[u]_2})$ with $A^{\omega}(J_{[u]_2})$, turns out to be exactly the isomorphism $(u, x) \rightarrow (f(u), x)$, where $f$ is the standard map on $[\overline{0}, \overline{1}] = A^{\omega}$. In this sense, standard maps are generalized instances of the classical Schroeder-Bernstein bijection in the context of orders of the form $A^{\omega}$. 

The standard map $f$ fixes the endpoints $\overline{0}$ and $\overline{1}$ of $A^{\omega}$---necessarily so, since $f$ is an order automorphism. It is easily checked that these are the only fixed points of $f$. This does not change the fact that $f$ is parity-reversing: both $\overline{0}$ and $\overline{1}$ are periodic sequences of period 1, so that $[a\overline{0}]_2 = [\overline{0}]_2 = [\overline{0}]$ and $[a\overline{1}]_2 = [\overline{1}]_2 = [\overline{1}]$ for any $a \in A$. Thus the parity-reversing requirement on $f$ at these points is simply that $f(\overline{0}) \in [\overline{0}]$ and $f(\overline{1}) \in [\overline{1}]$, which certainly holds. Viewing $f$ as defining an isomorphism between some $X \cong A^2X \cong A^{\omega}(I_{[u]_2})$ and $AX \cong A^{\omega}(J_{[u]_2})$, we have that this isomorphism maps $I_{\overline{0}}$ onto $J_{\overline{0}}$, which is legitimate as these intervals are identical, as noted in the discussion following \ref{switching}. Similarly for $I_{\overline{1}}$ and $J_{\overline{1}}$. 

If $A$ does not have any endpoints, or only a single endpoint, then $A^{\omega}$ is not of the form $[\overline{r}, \overline{s}]$. Thus no standard map defines a p.r.a. on all of $A^{\omega}$. However, because they fix the endpoints of the intervals on which they are defined, standard maps can be stitched together to obtain automorphisms of longer intervals. For example, suppose $r, s, t \in A^{<\omega}$ are sequences such that $\overline{r} < \overline{s} < \overline{t}$. Let $f$ denote the standard map on $[\overline{r}, \overline{s}]$ and $g$ the standard map on $[\overline{s}, \overline{t}]$. Then $f \cup g$ is an automorphism of $[\overline{r}, \overline{t}]$. Its fixed points are exactly $\overline{r}, \overline{s}$, and  $\overline{t}$. If $|r|, |s|$, and $|t|$ are all odd, then $f$ and $g$, and thus $f \cup g$, are parity-reversing. 

The standard map $h$ on $[\overline{r}, \overline{t}]$ also serves as an automorphism of this interval. This map is different than the two-piece map $f \cup g$ (e.g. $h$ does not fix $\overline{s}$). The advantage of the piecewise construction is that it can be extended to get parity-reversing maps on intervals without endpoints, as well as orders without endpoints. To deal with such intervals, we introduce some terminology. 

The \emph{cofinality} of an interval $I$ is the minimum length $\lambda$ of an increasing sequence of points $\{ x_i: i < \lambda \} \subseteq I$ such that for every $y \in I$ there is $i < \lambda$ with $y \leq x_i$. Similarly, the \emph{coinitiality} of $I$ is minimum length $\kappa$ of a decreasing sequence in $I$ that eventually goes below every element of $I$. The cofinality is 1 if the interval has a maximal element, and if it has a minimal element the coinitiality is 1. When these cardinals are not 1, they are infinite and regular. An interval with coinitiality $\kappa$ and cofinality $\lambda$ is called a $(\kappa, \lambda)$-interval. We usually write these as ordinals, e.g. $\omega, \omega_1,$ etc., to emphasize that they refer to sequences. We refer to $(1, 1)$-intervals as $2$-intervals, $(1, \omega)$-intervals as $\omega$-intervals, $(\omega, 1)$-intervals as $\omega^*$-intervals, and $(\omega, \omega)$-intervals as $\mathbb{Z}$-intervals, since these intervals are, respectively, spanned by sequences of order type $2$, $\omega$, $\omega^*$, and $\mathbb{Z}$. Viewing the order $A$ as itself an interval, we may speak of $(\kappa, \lambda)$-orders, or $2$, $\omega$, $\omega^*$, and $\mathbb{Z}$-orders. 

A \emph{cover} for an order $A$ is a collection of disjoint intervals $\mathcal{C}=\{C_a: a \in A\}$ so that $a \in C_a$ for all $a$. Indexing by the elements of $A$ is for convenience: it is not assumed that $a \neq b$ implies $C_a \neq C_b$, since this would give only trivial covers, but only that either $C_a = C_b$ or $C_a \cap C_b = \emptyset$. 

A cover $\mathcal{C}$ is called a $\mathbb{Z}$-cover if every $C \in \mathcal{C}$ is a $\mathbb{Z}$-interval. Not every order $A$ admits a $\mathbb{Z}$-cover: if, for example, $A$ has a left endpoint $0$, then in any cover $\mathcal{C}$ for $A$, the interval $C_0$ cannot be a $\mathbb{Z}$-interval. Less trivially, suppose $A$ is a complete ($\omega_1$, $\omega_1$)-order, and $\mathcal{C}$ is a cover for $A$. Suppose $\mathcal{C}$ contains a $\mathbb{Z}$-interval $C$. This interval is open. Since $A$ is ``long" to both the left and right, $C$ is bounded. Since $A$ is complete, $C$ has a greatest lower bound $x$ and least upper bound $y$. The intervals $C_x$ and $C_y$ must be disjoint from $C$, hence $x$ must be the maximal element in $C_x$ and $y$ the minimal element in $C_y$. Thus neither is a $\mathbb{Z}$-interval, and it follows that $A$ admits no $\mathbb{Z}$-cover. 

A cover $\mathcal{C}$ is called a $\{\mathbb{Z}, \omega\}$-cover if every $C \in \mathcal{C}$ is either a $\mathbb{Z}$-interval or an $\omega$-interval. Similarly, $\mathcal{C}$ is called a $\{\mathbb{Z}, \omega^*\}$-cover if every $C \in \mathcal{C}$ is either a $\mathbb{Z}$-interval or $\omega^*$-interval. Given a cover $\mathcal{C}$, let $\mathcal{C}_X \subseteq \mathcal{C}$ denote the collection of intervals in $\mathcal{C}$ of type $X$. For example, if $\mathcal{C}$ is a $\{\mathbb{Z}, \omega\}$-cover, then $\mathcal{C} = \mathcal{C}_{\mathbb{Z}} \cup \mathcal{C}_{\omega}$.

\theoremstyle{definition}
\newtheorem{Aonenonepra}[theorem]{Theorem}
\begin{Aonenonepra}\label{Aonenonepra}
\,\ \
\begin{enumerate}
\item[1.] If $A$ admits a $\mathbb{Z}$-cover, then $A^{\omega}$ has a parity-reversing automorphism.
\item[2.] If $A$ has a left endpoint and admits a $\{\mathbb{Z}, \omega\}$-cover, then $A^{\omega}$ has a parity-reversing automorphism.
\item[3.] If $A$ has a right endpoint and admits a $\{\mathbb{Z}, \omega^*\}$-cover, then $A^{\omega}$ has a parity-reversing automorphism.
\end{enumerate}
Thus in any of these three cases, if $X$ is an order such that $A^2X \cong X$, then $AX \cong X$ as well. 
\end{Aonenonepra}
\begin{proof}
(1.) Assume first that $A$ has a $\mathbb{Z}$-cover $\mathcal{C}$. It follows that $A$ has no endpoints. For every $C \in \mathcal{C}$, fix a $\mathbb{Z}$-sequence $\ldots < x^C_{-1} < x^C_0 < x^C_1 < x^C_2 < \ldots$ spanning $C$. To each of the points $x_k^C$ in $A$ there is the corresponding periodic sequence $\overline{x}_k^C$ in $A^{\omega}$. For $C, D \in \mathcal{C}$ and $k, l \in \mathbb{Z}$, the intervals $[x_k^C, x_{k+1}^C)$ and $[x_l^D, x_{l+1}^D)$ intersect if and only if they are identical. The same is true of the corresponding intervals $[\overline{x}_k^C, \overline{x}_{k+1}^C)$ and $[\overline{x}_l^D, \overline{x}_{l+1}^D)$ in $A^{\omega}$. 

These intervals actually cover $A^{\omega}$, that is, for every $u \in A^{\omega}$ there is a unique $C \in \mathcal{C}$ and $k \in \mathbb{Z}$ such that $u \in [\overline{x}_k^C, \overline{x}_{k+1}^C)$. To see this, note that since the intervals $C$ cover $A$, $u$'s first entry $u_0$ (viewed as a point in $A$) falls in one of the $C$. There are two possibilities: either $x_k^C < u_0 < x_{k+1}^C$ for some $k \in \mathbb{Z}$, or $u_0 = x_k^C$ for some $k \in \mathbb{Z}$. In the first case, we have that $\overline{x}_k^C < u < \overline{x}_{k+1}^C$ so that $u \in [\overline{x}_k^C, \overline{x}_{k+1}^C)$. In the second, either $\overline{x}_{k-1}^C < u < \overline{x}_k^C$ or $\overline{x}_k^C \leq u < \overline{x}_{k+1}^C$ depending on whether the first entry of $u$ differing from $x_k^C$ (if it exists) is greater or less than $x_k^C$. Thus either $u \in [\overline{x}_{k-1}^C, \overline{x}_{k}^C)$ or $u \in [\overline{x}_k^C, \overline{x}_{k+1}^C)$. 

We have shown that 
\[
A^{\omega} = \bigcup_{\substack{k \in \mathbb{Z} \\ C \in \mathcal{C}}} [\overline{x}_k^C, \overline{x}_{k+1}^C].
\]
The intervals in this union are pairwise disjoint unless they are consecutive, in which case they share a single endpoint.  

For every $k \in \mathbb{Z}$ and $C \in \mathcal{C}$, let $f_k^C$ denote the standard map on $[\overline{x}_k^C, \overline{x}_{k+1}^C]$. Since the endpoints of this interval are of period 1, we have that $f_k^C$ is parity-reversing.

The function 
\[ 
f = \bigcup_{\substack{k \in \mathbb{Z} \\ C \in \mathcal{C}}} f_k^C
\]
is well-defined, since the domains of two different $f_k^C$ share at most one point and at that point the functions agree. Since each $f_k^C$ is a parity-reversing automorphism of $[\overline{x}_k^C, \overline{x}_{k+1}^C]$ and $A^{\omega}$ is the union of these intervals, $f$ is a p.r.a. for $A^{\omega}$. \,\ \\

(2.) Now assume that $A$ has a left endpoint $0$ and a $\{\mathbb{Z}, \omega\}$-cover $\mathcal{C}$. Then $A$ has no right endpoint. The argument is similar to the previous one, but with a catch. For every $C \in \mathcal{C}_{\mathbb{Z}}$, fix a sequence $\ldots < x^C_{-1} < x^C_0 < x^C_1 < x^C_2 < \ldots$ spanning $C$. Similarly, for every $C \in \mathcal{C}_{\omega}$ fix a sequence $x^C_0 < x^C_1 < x^C_2 < \ldots$ spanning $C$, where $x_0^C$ is the left endpoint of $C$. 

As before, the intervals $[\overline{x}_k^C, \overline{x}_{k+1}^C)$ are pairwise disjoint. The difference is that now there may be points in $A^{\omega}$ that do not fall in any of these intervals.

To see this, suppose $u \in A^{\omega}$. Then there is a unique $C \in \mathcal{C}$ and $k$ such that $x_k^C \leq u_0 < x_{k+1}^C$. If in fact $x_k^C < u_0 < x_{k+1}^C$, then certainly $u \in [\overline{x}_k^C, \overline{x}_{k+1}^C)$. So suppose $u_0 = x_k^C$. If either $C \in \mathcal{C}_{\mathbb{Z}}$, or $C \in \mathcal{C}_{\omega}$ and $k > 0$, then $u$ is contained in either $[\overline{x}_{k-1}^C, \overline{x}_k^C)$ or $[\overline{x}_k^C, \overline{x}_{k+1}^C)$, depending on whether the first entry of $u$ differing from $x_k^C$ (if it exists) is greater than or less than $x_k^C$.  

So suppose $u_0 = x_0^C$ for some $C \in \mathcal{C}_{\omega}$. Then either $u \geq \overline{x}_0^C$ or $u < \overline{x}_0^C$. In the first case we have $u \in [\overline{x}_0^C, \overline{x}_1^C)$. 

The issue occurs in the second case when $u < \overline{x}_0^C$. Assume we are in this case, and for notational simplicity denote $x_0^C$ by $x$. It must be then, that $x$ (viewed as a point in $A$) is greater than the left endpoint $0$, and further that $u_n < x$, where $u_n$ is the leftmost entry of $u$ differing from $x$. However, since $u_0 = x$ it must be that $u \geq x000\ldots = x\overline{0}$. Thus $u \in [x\overline{0}, \overline{x})$. This interval is disjoint from all of the $[\overline{x}_k^C, \overline{x}_{k+1}^C)$, and by the same argument, any $v \in A^{\omega}$ that is not contained in one of the $[\overline{x}_k^C, \overline{x}_{k+1}^C)$ must be contained in an interval of this form. 

Thus
\[
A^{\omega} = \bigcup_{\substack{k \in \mathbb{Z} \\ C \in \mathcal{C}_{\mathbb{Z}}}} [\overline{x}_k^C, \overline{x}_{k+1}^C] \,\ \cup \bigcup_{\substack{k \in \mathbb{\omega} \\ C \in \mathcal{C}_{\omega}}} [\overline{x}_k^C, \overline{x}_{k+1}^C] \,\ \cup \bigcup_{\substack{x = x_0^C \\ C \in \mathcal{C} _{\omega}}} [x\overline{0}, \overline{x}].
\]
These intervals are pairwise disjoint up to endpoints. The intervals $[\overline{x}_k^C, \overline{x}_{k+1}^C]$ have parity-reversing standard maps $f_k^C$. If we can show that for every interval $[x\overline{0}, \overline{x}]$, there is a parity-reversing automorphism $f_x: [x\overline{0}, \overline{x}] \rightarrow [x\overline{0}, \overline{x}]$, then the map
\[
f = \bigcup_{\substack{k \in \mathbb{Z} \\ C \in \mathcal{C}_{\mathbb{Z}}}} f_k^C \,\ \cup \bigcup_{\substack{k \in \mathbb{\omega} \\ C \in \mathcal{C}_{\omega}}} f_k^C \,\ \cup \bigcup_{\substack{x = x_0^C \\ C \in \mathcal{C} _{\omega}}} f_x 
\]
is a parity-reversing automorphism of $A^{\omega}$. 

So fix $C \in \mathcal{C}_{\omega}$ and let $x = x_0^C$. If $C$ is the leftmost interval in $\mathcal{C}$, then $x = 0$ and $[x\overline{0}, \overline{x}] = \{\overline{0}\}$ is just the left endpoint of $A^{\omega}$. In this case let $f_x$ be the map that fixes $\overline{0}$ and is undefined elsewhere. 

Otherwise $x > 0$. The interval $[x\overline{0}, \overline{x}]$ has a periodic right endpoint, but only an eventually periodic left endpoint. We have not defined a standard map for such an interval. Note however that the shift map $u \mapsto \sigma u$ restricted to $[x\overline{0}, \overline{x}]$ is an isomorphism of this interval with $[\overline{0}, \overline{x}]$. Let $g$ denote this shift map, and let $f$ denote the standard map on $[\overline{0}, \overline{x}]$. Then $f_x = g^{-1} \circ f \circ g$ is a parity-reversing automorphism of $[x\overline{0}, \overline{x}]$, as desired.  

Now $f$, as defined above, is a p.r.a. for $A^{\omega}$.  \,\ \\

(3.) Finally, suppose $A$ has a right endpoint and admits a $\{\mathbb{Z}, \omega^*\}$-cover. Then $A^*$ has a left endpoint and admits a $\{\mathbb{Z}, \omega\}$-cover. Thus $(A^*)^{\omega}$, which is isomorphic to $(A^{\omega})^*$, has a p.r.a. $f^*$. The corresponding automorphism $f$ on $A^{\omega}$ is still parity-reversing since the requirement $f(u) \in [au]_2$ does not depend on the ordering of $A^{\omega}$, but only on the underlying set of points. 
\end{proof}

Suppose for the moment that $A$ is an order without endpoints. The content of the theorem is that if $A$ has a $\mathbb{Z}$-cover, then $A^{\omega}$ can be covered disjointly (up to endpoints) by intervals of the form $[\overline{x}, \overline{y}]$, with $x, y \in A$. Since we have parity-reversing standard maps on these intervals, we can use this decomposition to build a p.r.a. on all of $A^{\omega}$. 

The essence of what can go wrong when $A$ does not admit a $\mathbb{Z}$-cover can be illustrated by an example. Suppose that $A$ is a complete $(\omega_1, \omega_1)$-order. We have seen that $A$ has no $\mathbb{Z}$-cover. We might still attempt to build a p.r.a. for $A^{\omega}$ piecewise: begin with some $x_0 < x_1$ in $A$ and consider the corresponding periodic points $\overline{x}_0 < \overline{x}_1$ in $A^{\omega}$. Put the standard map on $[\overline{x}_0, \overline{x}_1]$. Then pick some $\overline{x}_2 > \overline{x}_1$, and so on. After $\omega$-many steps we will have defined a p.r.a. on the interval spanned by $\overline{x}_0 < \overline{x}_1 < \overline{x}_2 < \ldots$ . 

Now, since $A$ is complete and of cofinality $\omega_1$, the sequence $x_0 < x_1 < x_2 < \ldots$ is bounded in $A$ and converges to some point $x$. However, the sequence of $\overline{x}_n$ does \emph{not} converge to $\overline{x}$ in $A^{\omega}$. Rather, there is a nonempty open interval $I$ sitting above all the $\overline{x}_n$ and below $\overline{x}$. This interval $I$ consists of points $u$ with first coordinate $u_0=x$ but with some later coordinate $u_i < x$. Notice that there are no  periodic points of period 1 in this interval. We might hope to bridge this gap (perhaps up to some collection of legally fixable points) with intervals whose endpoints are only eventually periodic, as we did in the proof of (2.) of Theorem \ref{Aonenonepra}. There, however, it was crucial for the argument that $A$ had a left endpoint $0$.

It turns out it is impossible to bridge this gap, in the sense that any automorphism on $A^{\omega}$ extending the partial automorphism described above cannot be parity-reversing on $I$. Moreover, no alternative construction works. We will show in Section 6 that any complete $(\omega_1, \omega_1)$-order does not admit a parity-reversing automorphism. 

Here is an immediate corollary of Theorem \ref{Aonenonepra}.

\theoremstyle{definition}
\newtheorem{cntblwidth}[theorem]{Corollary}
\begin{cntblwidth}\label{cntblwidth}
If $A$ is a $2$-order, $\omega$-order, $\omega^*$-order, or $\mathbb{Z}$-order, then $A^{\omega}$ has a parity-reversing automorphism. Thus for any order $X$, if $A^2X \cong X$ then $AX \cong X$.
\end{cntblwidth}
\begin{proof}
When $A$ is a $2$-order, this is simply a restatement of Theorem \ref{Abothpra}. In the other three cases $A$ has, respectively, a $\{\mathbb{Z}, \omega\}$-cover, $\{\mathbb{Z}, \omega^*\}$-cover, or $\mathbb{Z}$-cover given by $\mathcal{C}=\{A\}$. 
\end{proof}

Among the orders satisfying the hypotheses of the corollary are all countable orders, since any countable order has a cofinality and coinitiality at most $\omega$. Thus for any countable $A$ and any $X$ we have $A^2X \cong X \implies AX \cong X$. This gives in particular that any countable $X$ isomorphic to its cube is isomorphic to its square, recovering Corollaries \ref{endpoints} and \ref{cntblnoend} and also giving the result for countable orders with a single endpoint. 

The order $\mathbb{R}$ of the real numbers gives an example of an uncountable $\mathbb{Z}$-order. Hence $\mathbb{R}^2 X \cong X \implies \mathbb{R}X \cong X$ for any $X$. One of the simplest examples of an order that is not a $\mathbb{Z}$-order but admits a $\mathbb{Z}$-cover is $\omega_1\mathbb{Z}$. Here, the cover consists of the intervals $C_{\alpha} = \{(\alpha, z): z \in \mathbb{Z} \}$. Similarly, $\omega_1  \mathbb{R}$ admits a $\mathbb{Z}$-cover, since each copy of $\mathbb{R}$ is spanned by a $\mathbb{Z}$-sequence. 

The order $\omega_1$ has a left endpoint, but admits no $\{\mathbb{Z}, \omega\}$-cover, as can be shown by a similar argument to the one showing any complete $(\omega_1, \omega_1)$-order has no $\mathbb{Z}$-cover. Thus it does not follow from \ref{Aonenonepra} that there is a p.r.a. on $\omega_1^{\omega}$. It turns out that $\omega_1^{\omega}$ does have a p.r.a., though we will not prove this. However, if $A = \omega^*_1 + \omega_1$ is the order obtained by putting a copy of $\omega_1$ to the right of a copy of $\omega^*_1$, then $A$ is a complete $(\omega_1, \omega_1)$-order. It follows from the results of Section 6 that there is no p.r.a. on $A^{\omega}$. We will show in fact that there exists an $X$ such that $A^2X \cong X$ but $AX \not \cong X$. 

It remains to show that any order $X$ that is isomorphic to its cube admits the right kind of cover to get a p.r.a. on $X^{\omega}$. This is Theorem \ref{Xcubedcovers} below. When combined with Theorem \ref{Aonenonepra}, it shows that for any $X$ without endpoints or with a single endpoint, if $X^3 \cong X$ then $X^2 \cong X$. Since we have already dealt with the two endpoint case, we have $X^3 \cong X \implies X^2 \cong X$ for any $X$, finishing the proof of the cube property. 

\theoremstyle{definition}
\newtheorem{Xcubedcovers}[theorem]{Theorem}
\begin{Xcubedcovers}\label{Xcubedcovers}
Suppose $X$ is an order such that $X^3 \cong X$. 
\begin{enumerate}
\item[1.] If $X$ has no endpoints, then $X$ admits a $\mathbb{Z}$-cover.
\item[2.] If $X$ has a left endpoint but no right endpoint, then $X$ admits a $\{\mathbb{Z}, \omega\}$-cover. 
\item[3.] If $X$ has a right endpoint but no left endpoint, then $X$ admits a $\{\mathbb{Z}, \omega^*\}$-cover. 
\end{enumerate}
\end{Xcubedcovers}

The proof of the theorem is in two lemmas. Before we can state these lemmas, we need some more terminology. 

Up to this point we have only studied invariance of orders under $\emph{left}$ multiplication (by a fixed $A$ or power of $A$). Orders $X$ such that $X^n \cong X$ fit into this context since the isomorphism $X^n \cong X$ can be rewritten as $A^{n-1}X \cong X$, where $A=X$. However, such orders also display an invariance under right multiplication, since $X^n \cong X$ can just as well be rewritten $XA^{n-1} \cong X$, where again $A=X$. It is this right-sided invariance under multiplication that is needed to get the covers of $X$ in Theorem \ref{Xcubedcovers}.

Forgetting the isomorphism $X^3 \cong X$ for now, we consider the single power version of this right-sided invariance. That is, for an order $A$, we analyze the structure of orders $X$ such that $XA \cong X$. 

In the case of invariance under left multiplication by $A$, the space $A^{\omega}$ of right-infinite sequences plays a crucial role. In the case of right-sided invariance, it is the space of left-infinite sequences that is relevant.

Let $A^{\omega^*} = \{(\ldots, a_2, a_1, a_0): a_i \in A, i \in \omega\}$ denote the set of left-infinite sequences of elements of $A$. Notice that we still index the entries of such sequences by elements in $\omega$. We copy over the notation from the right-sided case, reversing it when necessary. The letters $u, v, \ldots$ will now be used to denote elements of $A^{\omega^*}$. The $n$th entry of $u$ is still denoted $u_n$. 

Let $A^{<\omega^*} = \{(a_{n-1}, \ldots, a_1, a_0): a_i \in A, n \in \omega\}$ denote the set of left-growing finite sequences. The letters $r, s, \ldots$ will for now denote elements of $A^{<\omega^*}$. For $r = (a_{n-1}, \ldots, a_1, a_0) \in A^{<\omega^*}$ and $u = (\ldots, u_1, u_0) \in A^{\omega^*}$, we use $ur$ to denote the sequence $(\ldots, u_1, u_0, a_{n-1}, \ldots, a_1, a_0)$. 

It is impossible to order $A^{\omega^*}$ lexicographically, in the sense that two left-infinite sequences may not have a leftmost place in which they differ. If, however, two such sequences eventually agree, it is possible to compare them lexicographically. 

\theoremstyle{definition}
\newtheorem{eventequiv}[theorem]{Definition}
\begin{eventequiv}\label{eventequiv}
For $u, v \in A^{\omega^*}$, we say $u$ is \emph{eventually equal} to $v$, and write $u \sim_{\infty} v$, if there exists $N \in \omega$ such that for all $n > N$ we have $u_n = v_n$. 
\end{eventequiv}

Equivalently, $u \sim_{\infty} v$ if there exist finite sequences $r, s$ with $|r| = |s|$ and a sequence $u' \in A^{\omega^*}$ such that $u = u'r$ and $v = v's$. This is just the usual eventual equality relation, considered here for left-infinite sequences. The $\sim_{\infty}$-class of $u$ is denoted $[u]_{\infty}$. 

Observe that for a given $u \in A^{\omega^*}$, the class $[u]_{\infty}$ can be ordered lexicographically. If $u \sim_{\infty} v$, define $u < v$ if and only if $u_N < v_N$, where $N$ is the leftmost place such that $u_N \neq v_N$. It is immediate that $u \not < u$, and one of $u < v$, $u = v$, $u > v$ always holds. Transitivity can be checked as well.

The classes $[u]_{\infty}$ are the largest subsets of $A^{\omega^*}$ that can possibly be ordered lexicographically, in the sense that if $u \not \sim_{\infty} v$, then $u$ and $v$ have no leftmost place of difference. In what follows, whenever we refer to an order on $[u]_{\infty}$ we mean the lexicographical order. If we write $u < v$ it is assumed that $u \sim_{\infty} v$.

Here are our lemmas. 

\theoremstyle{definition}
\newtheorem{uinftypes}[theorem]{Lemma}
\begin{uinftypes}\label{uinftypes}
Fix $u \in A^{\omega^*}$. 
\begin{enumerate}
\item[1.] If $A$ has no endpoints, the class $[u]_{\infty}$ is a $\mathbb{Z}$-order.
\item[2.] If $A$ has a left endpoint but no right endpoint, the class $[u]_{\infty}$ is either an $\omega$-order or $\mathbb{Z}$-order. 
\item[3.] If $A$ has a right endpoint but no left endpoint, the class $[u]_{\infty}$ is either an $\omega^*$-order or $\mathbb{Z}$-order. 
\end{enumerate} 
\end{uinftypes}

\begin{proof}
(1.) Assume first $A$ has no endpoints. We define a sequence $\ldots < v^{-1} < v^0 < v^1 < \ldots$ spanning $[u]_{\infty}$. Let $v^0 = u$. For $n$ a fixed positive integer, define $v^n = (\ldots, v^n_1, v^n_0)$ to be any sequence such that $v^n_m = u_m$ for all $m > n$ but $v^n_n > u_n$. It is always possible to find such a $v^n_n$, since $A$ does not have a top point. On the other side, again for $n$ a fixed positive integer, let $v^{-n}$ be a sequence such that $v^{-n}_m=u_m$ for all $m > n$, but now $v^{-n}_n < u_n$. This is possible since $A$ does not have a bottom point. 

It is clear that $\ldots < v^{-1} < v^0 < v^1 < \ldots$ since the ordering is lexicographical. Further if $v \in [u]_{\infty}$, then there is some $N$ such that for all $m \geq N$ we have $v_m = u_m$. Thus $v^{-N} < v < v^N$, and so the sequence $v^n$ spans $[u]_{\infty}$, as desired. \,\ \\

(2.) Now assume $A$ has a left endpoint $0$ but no right endpoint. The class $[\overline{0}]_{\infty}$ has a left endpoint, namely $\overline{0}$ itself. Let $v^0 = \overline{0}$ and $v^n$ be any sequence that is 0 beyond the $n$th place, but is greater than $0$ in the $n$th place. Then we have $v^0 < v^1 < \ldots$ as before, and in fact this sequence spans $[\overline{0}]_{\infty}$. Thus $[\overline{0}]_{\infty}$ is an $\omega$-order. 

So assume $u \not \sim_{\infty} \overline{0}$. We show $[u]_{\infty}$ is a $\mathbb{Z}$-order. As before, since $A$ has no top point, we may find $v^1 < v^2 < \ldots$ cofinal in $[u]_{\infty}$. 

On the other side, observe that since $u \not\sim_{\infty} \overline{0}$, there are infinitely many places $n$ such that $u_n > 0$. Enumerate these places as $n_k$, $k \in \omega$. Let $v^{n_k}$ be any sequence that agrees with $u$ beyond the $n_k$th place, but such that $v^{n_k}_{n_k} < u_{n_k}$. This is possible since $u_{n_k} > 0$. Then $\ldots < v^{n_1} < v^{n_0}$, and further this sequence is coinitial in $[u]_{\infty}$. Hence $[u]_{\infty}$ is a $\mathbb{Z}$-order, as claimed. \,\ \\

The case (3.) is symmetric to (2.). 
\end{proof}

\theoremstyle{definition}
\newtheorem{XAcovers}[theorem]{Lemma}
\begin{XAcovers}\label{XAcovers}
Suppose that $X$ is an order such that $XA \cong X$. 
\begin{enumerate}
\item[1.] If $A$ has neither a left nor right endpoint, then $X$ admits a $\mathbb{Z}$-cover.
\item[2.] If $A$ has a left endpoint but no right endpoint, then $X$ admits a $\{\mathbb{Z}, \omega\}$-cover. 
\item[3.] If $A$ has a right endpoint but no left endpoint, then $X$ admits a $\{\mathbb{Z}, \omega^*\}$-cover. 
\end{enumerate}
\end{XAcovers}

\begin{proof}
The intuition of the proof is simple. If $XA \cong X$, then $X$ can be organized into $X$-many intervals of type $A$. Since there are $X$-many of them, these intervals in turn may be organized into $XA$-many copies of $A$, that is, $X$-many copies of $A^2$. And so on. Now consider the situation from the point of view of some fixed $x \in X$. At the first stage $x$ is included in a copy of $A$, and in the second, in some copy of $A^2$ containing the initial copy of $A$, etc. These larger and larger intervals surrounding $x$, consecutively isomorphic to $A$, $A^2$, $A^3$, $\ldots$, have a limit, which is isomorphic to $[u]_{\infty}$ for some $u \in A^{\omega^*}$. The conclusion follows. 

To make this explicit, let $f: X \rightarrow XA$ be an isomorphism. Let $f_l$ and $f_r$ denote the left and right components of $f$, that is, the unique functions such that $f(x) = (f_l(x), f_r(x))$ for all $x \in X$. For every $x \in X$, we define a sequence of points $x_0, x_1, \ldots$ in $X$ and a sequence $a_0^x, a_1^x, \ldots$ in $A$. Let $x_0 = x$ and recursively define $x_{n+1} = f_l(x_n)$. Let $a_n^x = f_r(x_n)$. 

With this notation, we have $f(x)=(x_1, a_0^x)$. By repeatedly factoring the $x_i$, we get an isomorphism from $X$ onto $XA^n$ defined by $x \mapsto (x_n, a_{n-1}^x, \ldots a_1^x, a_0^x)$. Although it is not literally an $n$-fold composition of $f$, we denote this isomorphism by $f^n$. 

For $n \in \omega$, define $I_n^x$ to be the set of $y \in X$ such that $y_n = x_n$. Then the image of $I_n^x$ under $f^n$ is the set of points in $XA^n$ of the form $(x_n, b_{n-1}, \ldots, b_0)$, $b_i \in A$. Hence $I_n^x$ is an interval in $X$, and it is isomorphic to $A^n$. Furthermore, we have the containments $\{x\} = I_0^x \subseteq I_1^x \subseteq I_2^x \subseteq \ldots$ since if $y_N = x_N$ then for every $n \geq N$ we have $y_n = x_n$ as well.

Let $I_{\infty}^x = \bigcup_n I_n^x$. Then since the $I_n^x$ form a chain of intervals in $X$, $I_{\infty}^x$ is also an interval in $X$. Notice that by definition, for every $N\in \omega$ and $x, y \in X$, either $I_N^x \cap I_N^y = \emptyset$ or $I_N^x = I_N^y$, and in this latter case $I_n^x = I_n^y$ for all $n \geq N$. Hence $I_{\infty}^x$ and $I_{\infty}^y$ are either equal or disjoint as well. 

For $x \in X$, define $u_x$ to be the sequence $(\ldots, a_1^x, a_0^x)$. Then if $y \in I_{\infty}^x$, it must be that for all sufficiently large $n$ we have $y_n = x_n$. Thus for all sufficiently large $n$ we have $a_n^x=a_n^y$, which gives $u_x \sim_{\infty} u_y$. On the other hand, suppose $v \in [u_x]_{\infty}$, say $v = (\ldots, a_{n+1}^x, a_n^x, b_{n-1}, \ldots, b_1, b_0)$ for some $b_i \in A$, $i <n$. There is a unique $y \in X$ such that $f^n(y) = (y_n, a^y_{n-1}, \ldots, a^y_0) = (x_n, b_{n-1}, \ldots, b_0)$. But then $u_y = v$. 

This shows that the map $F: I_{\infty}^x \rightarrow [u_x]_{\infty}$ defined by $F(y) = u_y$ is a bijection. This map is clearly order-preserving as well, and hence an isomorphism of $I_{\infty}^x$ with $[u_x]_{\infty}$. Thus $\{I_{\infty}^x: x \in X \}$ is a cover of $X$ by intervals of the form $[u]_{\infty}$. The conclusion of the lemma now follows from Lemma \ref{uinftypes}. 
\end{proof}

\begin{proof}[Proof of Theorem \ref{Xcubedcovers}]
Suppose $X$ is an order such that $X^3 \cong X$. Then $XA \cong X$, where $A = X^2$. If $X$ has no endpoints, then $A$ has no endpoints. By Lemma \ref{XAcovers}, $X$ admits a $\mathbb{Z}$-cover, and therefore, by Lemma \ref{uinftypes}, $X^{\omega}$ admits a parity-reversing automorphism. If $X$ has a left endpoint, but not a right one, then similarly $A$ has a left endpoint, but no right one. Hence $X$ admits a $\{\mathbb{Z}, \omega\}$-cover, and again $X^{\omega}$ admits a parity-reversing automorphism. The right endpoint case is symmetric. 
\end{proof}

We conclude this section with a sketch of the proof that $X^n \cong X$ implies $X^2 \cong X$ for all orders $X$ and all $n \geq 2$. When $n=2$ the statement is trivial, and we have just finished the proof for $n=3$. The argument for larger $n$ is a straightforward adaptation of the case when $n=3$.

For convenience, we consider orders satisfying the (only notationally distinct) isomorphism $X^{n+1} \cong X$, and we assume $n>2$. This isomorphism can be rewritten as $A^nX \cong X$, where $A=X$. We know by the results at the end of Section 3 that if $A^{\omega}$ has an $n$-revolving automorphism ($n$-r.a.), then $AX \cong X$. So we turn to the question of when $A^{\omega}$ admits such an automorphism. 

We build $n$-r.a.'s as we built p.r.a.'s, that is, as unions of standard maps. In our previous argument, to ensure that a map $f$ was parity-reversing, it was enough to have that for every $u$ either $f(u) = \sigma^n u$ or $f(u) = ru$, with $n$ and $|r|$ both odd. This is because deleting or adding an initial sequence of odd length always sends $u$ into $[au]_2$. For $n > 2$, the situation is not symmetric. 

\theoremstyle{definition}
\newtheorem{nra}[theorem]{Lemma}
\begin{nra}\label{nra}
Suppose $f: A^{\omega} \rightarrow A^{\omega}$ is an order automorphism, and for every $u \in A^{\omega}$, either $f(u) = \sigma^k u$ for some $k \equiv n-1 \pmod n$ or there exists a finite sequence $r$ with $|r| \equiv 1 \pmod n$ and $f(u) = ru$. Then $f$ is an $n$-r.a.
\end{nra}
\begin{proof}
In either case, the obvious meeting representation witnesses $f(u) \sim_n au$. 
\end{proof}

Thus if we have two sequences $r, s \in A^{<\omega}$ such that $|r| \equiv n-1 \pmod n$ and $|s| \equiv 1 \pmod n$, the standard map on $[\overline{r}, \overline{s}]$ is $n$-revolving. As before, when $A$ has both a left and right endpoint, a single standard map serves to get an $n$-r.a. of $A^{\omega}$. 

\theoremstyle{definition}
\newtheorem{Abothnra}[theorem]{Theorem}
\begin{Abothnra}\label{Abothnra}
Suppose $A$ has both a left and right endpoint. Then $A^{\omega}$ has an $n$-revolving automorphism. 
\end{Abothnra}
\begin{proof}
Let $0$ and $1$ denote the left and right endpoints of $A$. Then $A^{\omega} = [\overline{0^{n-1}}, \overline{1}]$. By Lemma \ref{nra}, the standard map on this interval is $n$-revolving. 
\end{proof}

In particular, if $X$ has both endpoints and $X^{n+1} \cong X$, then $X^2 \cong X$. For the cases with one or neither endpoint, we need the generalizations of Theorems \ref{Aonenonepra} and \ref{Xcubedcovers}. 

\theoremstyle{definition}
\newtheorem{Aonenonenra}[theorem]{Theorem}
\begin{Aonenonenra}\label{Aonenonenra}
\,\ \
\begin{enumerate}
\item[1.] If $A$ admits a $\mathbb{Z}$-cover, then $A^{\omega}$ has an $n$-r.a.
\item[2.] If $A$ has a left endpoint and admits a $\{\mathbb{Z}, \omega\}$-cover, then $A^{\omega}$ has an $n$-r.a.
\item[3.] If $A$ has a right endpoint and admits a $\{\mathbb{Z}, \omega^*\}$-cover, then $A^{\omega}$ has an $n$-r.a.
\end{enumerate}
Thus in any of these three cases, if $X$ is an order such that $A^nX \cong X$, then $AX \cong X$ as well. 
\end{Aonenonenra}

\theoremstyle{definition}
\newtheorem{Xnplusonecovers}[theorem]{Theorem}
\begin{Xnplusonecovers}\label{Xnplusonecovers}
Suppose $X$ is an order such that $X^{n+1} \cong X$. 
\begin{enumerate}
\item[1.] If $X$ has neither a left nor right endpoint, then $X$ admits a $\mathbb{Z}$-cover.
\item[2.] If $X$ has a left endpoint but no right endpoint, then $X$ admits a $\{\mathbb{Z}, \omega\}$-cover. 
\item[3.] If $X$ has a right endpoint but no left endpoint, then $X$ admits a $\{\mathbb{Z}, \omega^*\}$-cover. 
\end{enumerate}
\end{Xnplusonecovers}

The conjunction of these theorems along with \ref{Abothnra} gives that $X^{n+1} \cong X \implies X^2 \cong X$ for all $X$. Theorem \ref{Xnplusonecovers} follows immediately from Lemmas   \ref{uinftypes} and \ref{XAcovers}. The proof of \ref{Aonenonenra} (1.) is essentially the same as \ref{Aonenonepra} (1.): if $A$ has a $\mathbb{Z}$-cover, then $A^{\omega}$ can be covered (disjointly up to endpoints) by intervals of the form $[\overline{a}, \overline{b}]$ for $a, b \in A$. Writing such intervals as $[\overline{a^{n-1}}, \overline{b}]$, we have that the associated standard maps are $n$-revolving. The union of these maps then gives an $n$-r.a. on $A^{\omega}$. 

The proof of \ref{Aonenonenra} (2.) is also very similar to that of \ref{Aonenonepra} (2.). Given a $\{\mathbb{Z}, \omega\}$-cover of $A$ we obtain a cover of $A^{\omega}$ by intervals of the form $[\overline{a}, \overline{b}]$ and $[x\overline{a}, \overline{b}]$. Intervals of the first type have $n$-revolving standard maps, as above. Intervals of the second type shift onto intervals of the first type, and therefore also have $n$-revolving automorphisms. The union of these maps yields an $n$-r.a. for $A^{\omega}$. 

The proof of (3.) is symmetric. The main theorem follows: 

\theoremstyle{definition}
\newtheorem{mainthm}[theorem]{Theorem}
\begin{mainthm}\label{mainthm}
Suppose $X$ is an order such that $X^n \cong X$ for some $n>1$. Then $X^2 \cong X$. In particular, the cube property holds for $(LO, \times)$. 
\end{mainthm}

\section{Constructing orders $X$ such that $X^n \cong X$} 

The purpose of this section is to justify the previous work, in two ways. First, we will show that the cube problem for linear orders, and more generally the problem of showing $X^n \cong X \implies X^2 \cong X$, is not vacuous, in the sense that for every $n$ there exist many orders $X$ such that $X^n \cong X$. We construct such orders below as direct limits, and show in particular that they can be of any infinite cardinality. 

Secondly, we wish the justify the need for the machinery developed in Section 4 for solving the cube problem. When we say ``$X^3 \cong X$," what is meant implicitly is that there exists an isomorphism $f: X^3 \rightarrow X$. All of our analysis of the relation $X^3 \cong X$ has really been with respect to a fixed isomorphism $f$. Associated to such an $f$ is an order of the form $X^{\omega}(I_{[u]_2})$ and an isomorphism $F: X \rightarrow X^{\omega}(I_{[u]_2})$ built using $f$. If we view $f$, via the relabeling $F$, as an isomorphism of $X^2 \times X^{\omega}(I_{[u]_2})$ with $X^{\omega}(I_{[u]_2})$, then $f$ is just the flattening isomorphism $fl_2$ on the first three coordinates. This follows from the proof of \ref{thmax} and the discussion afterwards. Conversely, if we can ever construct an isomorphism $F$ from $X$ onto an order of the form $X^{\omega}(I_{[u]_2})$, we immediately have that $X^3 \cong X$ as witnessed by the flattening isomorphism.

Suppose it were the case that for every $X$ for which there exists an isomorphism $F$ of $X$ onto an order of the form $X^{\omega}(I_{[u]_2})$, we had that $I_{[u]_2} \cong I_{[au]_2}$ for every $a \in X$ and $u \in X^{\omega}$. Let us call such a decomposition \emph{trivial}. Then letting $I_{[u]}$ denote the common order type of $I_{[u]_2}$ and $I_{[au]_2}$, we would have that $X \cong X^{\omega}(I_{[u]})$ and hence $X^2 \cong X$. We would have this isomorphism of $X$ with $X^2$ without any need for a parity-reversing automorphism of $X^{\omega}$, and the work in Section 4 showing that $X^{\omega}$ has such an automorphism would be unnecessary to solve the cube problem. However, we shall show that this is not the case. There are pairs $(X, F)$ where $X$ is a linear order and $F$ is an isomorphism of $X$ onto an order of the form $X^{\omega}(I_{[u]_2})$ such that for many $a$ and $u$ we have $I_{[au]_2} \not\cong I_{[u]_2}$. This follows from Theorem \ref{XsquaredXhard} below. For such an $X$ we have $X^3 \cong X$ naturally, but in order to show $X^2 \cong X$ we need the device of a p.r.a. for $X^{\omega}$. 

Although Theorem \ref{XsquaredXsimple} can be derived from \ref{XsquaredXhard}, it has a proof which is easier to understand and thereby serves as a warmup for the proof of \ref{XsquaredXhard}. There are proofs of these theorems in more algebraic language, but we elect to keep the presentation elementary.

Given orders $X$ and $Y$, an \emph{embedding} of $X$ into $Y$ is an injective order-preserving map $f: X \rightarrow Y$. Given embeddings $f: X_0 \rightarrow Y_0$ and $g: X_1 \rightarrow Y_1$ the map defined by $(x, y) \mapsto (f(x), g(y))$ is an embedding of $X_0 \times X_1$ into $Y_0 \times Y_1$. We denote this map by $(f, g)$. 

Given a sequence of orders $X_0 \subseteq X_1 \subseteq X_2 \subseteq \ldots$ the order $X = \bigcup_{i \in \omega} X_i$ is well-defined. Similarly, given a directed system 
\[
X_0 \xrightarrow{f_0} X_1 \xrightarrow{f_1} \ldots
\]
where each $f_i$ is an embedding of $X_i$ into $X_{i+1}$, we may form the direct limit
\[
X = \varinjlim_{i} X_i.
\]
If the $f_i$ are inclusion maps, the direct limit of the $X_i$ is isomorphic to their union. We will sometimes confuse the direct limit construction with the union construction, and speak of each $X_i$ as a suborder of $X_j$ for $j\geq i$, and as a suborder of $X = \varinjlim X_i$. 

Given two systems
\[
X_0 \xrightarrow{f_0} X_1 \xrightarrow{f_1} \ldots \] \[
Y_0 \xrightarrow{g_0} Y_1 \xrightarrow{g_1} \ldots
\]
we obtain a system
\[
X_0 \times Y_0 \xrightarrow{(f_0, g_0)} X_1 \times Y_1 \xrightarrow{(f_1, g_1)} \ldots
\]
It is a standard fact that if $Z = \varinjlim (X_i \times Y_i)$, then $Z \cong X \times Y$, where $X = \varinjlim X_i$ and $\varinjlim Y_i$. Written shortly, we have
\[
\varinjlim (X_i \times Y_i) \cong (\varinjlim X_i) \times (\varinjlim Y_i).
\] 

Given orders $X$ and $Y$, we say that $X$ \emph{spans} $Y$ if there is an embedding $f$ of $X$ into $Y$ such that for every $y \in Y$, there exist $y_0, y_1$ in the image $f[X]$ such that $y_0 \leq y \leq y_1$. The following theorem says that any order can be expanded to an order of the same cardinality that is isomorphic to its square. 

\theoremstyle{definition}
\newtheorem{XsquaredXsimple}[theorem]{Theorem}
\begin{XsquaredXsimple}\label{XsquaredXsimple}
Let $X_0$ be any order. Then there exists an order $X$ such that
\begin{enumerate}
\item[1.] $X^2 \cong X$, 
\item[2.] $X_0$ spans $X$,
\item[3.] $|X| = |X_0| + \aleph_0$. 
\end{enumerate}
\end{XsquaredXsimple}

\begin{proof}
For every $i \in \omega$, let $X_{i+1} = X_i \times X_i$. Then $X_n = X_0^{2^n}$. Let $f_0: X_0 \rightarrow X_1$ be the embedding defined by $f_0(x) = (x, x)$. For every $i > 0$, let $f_i = (f_{i-1}, f_{i-1})$. Then $f_i$ is an embedding of $X_i$ into $X_{i+1}$. 

Let $X = \varinjlim X_i$. Then we have
\[
\begin{array}{l l l l}
X & = & \varinjlim X_i & (i \geq 0) \\
& \cong & \varinjlim X_i & (i > 0) \\
& = & \varinjlim (X_{i-1} \times X_{i-1}) & (i > 0) \\
& = & \varinjlim (X_i \times X_i) & (i \geq 0) \\
& \cong & \varinjlim(X_i) \times \varinjlim(X_i) & (i \geq 0)\\
& = & X \times X. &
\end{array}
\]
This proves (1.). The cardinality claim (3.) is clear. To verify (2.), let us view the $f_i$ as inclusions, so that each $X_i$ is included in $X_j$ for $j \geq i$, and in $X$. Then $X_0$ is included in $X_1 = X_0 \times X_0$ as the set of points of the form $(a, a)$, and more generally in $X_n = X_0^{2^n}$ as the set of points of the form $(a, a, \ldots, a)$. 

Fix $x \in X$. Then $x \in X_n$ for some $n$. Clearly, for some $a, b$ we have $(a, a, \ldots, a) \leq x \leq (b, b, \ldots, b)$. That is, $x$ lies between two points of $X_0$, provided we view $X_0$ as a subset of $X_n$. Hence we also have that $x$ lies between two points of $X_0$, now viewing $X_0$ as a subset of $X$. 
\end{proof}

Note in particular that $X$ has the same endpoint configuration as $X_0$: if $X_0$ has both endpoints, then so does $X$, and likewise for the other cases. Hence there are orders of any cardinality and any endpoint configuration isomorphic to their squares. 

We can get ``$X^3 \cong X$" instead of ``$X^2 \cong X$" in the conclusion of the theorem by letting $X_{i+1} = X_i^3$ and letting $f_0$ be the embedding $x \mapsto (x, x, x)$. Similarly, we can get orders $X$ such that $X^n \cong X$. 

However, with this particular construction the resulting isomorphisms turn out to be trivial. That is, if one uses the proof of \ref{XsquaredXsimple} to produce an order $X$ isomorphic to $X^3$, and then analyzes the isomorphism $f: X^3 \rightarrow X$ yielded by the proof and the associated decomposition $X^{\omega}(I_{[u]_2})$, one finds that $I_{[u]_2} = 1$ if and only if $u$ is eventually constant. Otherwise $I_{[u]_2} = \emptyset$. Since for eventually constant $u$ we have $[u]_2 = [au]_2 = [u]$, such a decomposition already witnesses $X^2 \cong X$ without the need for a p.r.a. of $X^{\omega}$. 

We shall now show how to directly build an order $X$ along with an isomorphism $F$ of $X$ onto an order of the form $X^{\omega}(I_{[u]})$. More generally we can construct $X$ and $F$ with $F: X \rightarrow X^{\omega}(I_{[u]_n})$. It will follow from the construction that these decompositions can be arranged to be non-trivial.

We have observed that if the embedding from $X_i \times Y_i$ into $X_{i+1} \times Y_{i+1}$ is of the form $(f_i, g_i)$, we have that the limit of the products $X_i \times Y_i$ is isomorphic to the product of the limits. An analogous fact holds for limits of replacements, as well as for limits of infinite products. Let us spell these out. 

Suppose we are given a system 
\[
X_0 \xrightarrow{f_0} X_1 \xrightarrow{f_1} \ldots
\]
Let $X = \varinjlim X_i$. We again view each $X_i$ as included in all subsequent $X_j$ as well as in $X$. Suppose further that for every $x \in X$ we are given a system 
\[
I_{0, x} \xrightarrow{g_{0, x}} I_{1, x} \xrightarrow{g_{1,x}} \ldots
\]
For each $x \in X$, let $I_x = \varinjlim I_{i, x}$.

Naturally we obtain a system
\[
X_0(I_{0, x}) \xrightarrow{F_0} X_1(I_{1, x}) \xrightarrow{F_1} \ldots
\]
The embedding $F_i$ is defined by $F_i(x, y) = (f_i(x), g_{i, x}(y))$, and in the replacement $X_i(I_{i, x})$, it is understood that the index $x$ of each $I_{i, x}$ only ranges over $X_i$. If we let $Y = \varinjlim X_i(I_{i, x})$, then one may verify that $Y \cong X(I_x)$. 

Similarly, if we are given 
\[
X_0 \xrightarrow{f_0} X_1 \xrightarrow{f_1} \ldots
\]
then we obtain a system
\[
X_0^{\omega} \xrightarrow{F_0} X_1^{\omega} \xrightarrow{F_1} \ldots
\]
where the embedding $F_i$ is defined by $F_i((x_0, x_1, \ldots)) = (f_i(x_0), f_i(x_1), \ldots)$. We use the notation $F_i = (f_i, f_i, \ldots)$. Viewing each $X_i$ as included in $X_{i+1}$ by way of $f_i$, we may view $X_{i}^{\omega}$ as included in $X_{i+1}^{\omega}$ by way of $F_i$. Letting $Y = \varinjlim X_i^{\omega}$, observe that it is not the case (naturally, at least) that $Y$ is isomorphic to $X^{\omega}$, where $X = \varinjlim X_i$. Rather, $Y$ is isomorphic to the subset of $X^{\omega}$ consisting of sequences of ``bounded rank," that is, sequences $(x_0, x_1, \ldots)$ such that for some $i$, for all $n$ we have $x_n \in X_i$. If the $F_i$ are true inclusions then $Y = \bigcup_i X_i^{\omega}$. Note that $Y$ is a union of tail-equivalence classes. 

\theoremstyle{definition}
\newtheorem{XsquaredXhard}[theorem]{Theorem}
\begin{XsquaredXhard}\label{XsquaredXhard}
Let $\{L_j: j \in J\}$ by any collection of nonempty, pairwise non-isomorphic orders, indexed by some indexing set $J$. Then there exists an $X$ such that $X \cong X^{\omega}(I_{[u]})$ for some collection of orders $I_{[u]}$ (and hence $X^2 \cong X$), and further such that
\begin{enumerate}
\item[1.] For $u, v \in X^{\omega}$, if $[u] \neq [v]$ then either $I_{[u]} = I_{[v]} = \emptyset$ or $I_{[u]}$ and $I_{[v]}$ are non-isomorphic. 
\item[2.] For every $j \in J$, there exists a unique tail-equivalence class $[u]$ such that $I_{[u]} = L_j$. 
\end{enumerate}
\end{XsquaredXhard}
\begin{proof}
Let $X_0$ be any order such that the number of tail-equivalence classes in $X_0^{\omega}$ is at least $|J|$. 

Fix an injection $\iota: J \rightarrow \{[u]: u \in X_0^{\omega}\}$. If $\iota(j) = [u]$, define $I_{0, [u]} = L_j$. For those $[u]$ not assigned an index $j \in J$, pick orders $I_{0, [u]}$ so that the final collection $\{I_{0, [u]}: u \in X_0^{\omega}\}$ consists of pairwise non-isomorphic orders. 

Denote the order $X_0^{\omega}(I_{0, [u]})$ by $X_1$. Fix an embedding $f_0: X_0 \rightarrow X_1$. For example, we may define $f_0$ by $f_0(x) = (x, x, x, \ldots, a_x)$, where $a_x$ is any element in $I_{0, [\overline{x}]}$. 

By way of this $f_0$, view $X_0$ as included in $X_1$. Then as in the discussion preceding the theorem, we may view $X_0^{\omega}$ as included in $X_1^{\omega}$ by way of $F_0 = (f_0, f_0, \ldots)$. For each $u \in X_1^{\omega}$, if $u \in X_0^{\omega}$, define $I_{1, [u]} = I_{0, [u]}$. The remaining $u$ are exactly those sequences with infinitely many terms from $X_1 \setminus X_0$. For these $[u]$, iteratively choose orders $I_{1, [u]}$ so that the final collection $\{I_{1, [u]}: u \in X_1^{\omega}\}$ consists of pairwise non-isomorphic orders. 

Let $X_2 = X_1^{\omega}(I_{1, [u]})$. The ``inclusion" $F_0: X_0^{\omega} \rightarrow X_1^{\omega}$ naturally determines an ``inclusion" $f_1: X_1 \rightarrow X_2$. Namely, if $(u, a) \in X_1 = X_0^{\omega}(I_{0, [u]})$, with say $u = (u_0, u_1, \ldots)$, we define $f_1((u, a)) = (F_0(u), a)=(f_0(u_0), f_0(u_1), \ldots, a)$. It is legal to let $f_1$ be the identity on the last coordinate, since the interval $(u, \cdot)$ in $X_1$ is of type $I_{0, [u]}$, and the interval $(F_0(u), \cdot)$ in $X_2$ is of type $I_{1, [u]} = I_{0, [u]}$. 

Now repeat this process. View $X_1$ as included (by way of $f_1$) in $X_2$, and $X_1^{\omega}$ as included in $X_2^{\omega}$. For each $u \in X_2^{\omega}$, if $u \in X_1^{\omega}$, let $I_{2, [u]} = I_{1, [u]}$. For the remaining $u$, fix orders $I_{2, [u]}$ so that the collection $\{I_{2, [u]}: u \in X_2^{\omega}\}$ consists of pairwise non-isomorphic orders. 

Continuing in this way, we get a system
\[
X_0 \xrightarrow{f_0} X_1 \xrightarrow{f_1} X_2 \xrightarrow{f_2} \ldots
\]
Let $X$ be the limit. Then we have
\[
\begin{array}{l l l}
X & = & \varinjlim X_i \\
& \cong & \varinjlim X_{i+1} \\
& = & \varinjlim X_i^{\omega}(I_{i, [u]}) \\
& = & X^{\omega}(I_{[u]})  
\end{array}
\]
where if $u \in X_i^{\omega}$ for some $i$, then $I_{[u]} = I_{i, [u]}$, and if $u$ is in none of the $X_i^{\omega}$ (i.e. if $u$ has terms of unboundedly high rank), then $I_{[u]} = \emptyset$. 

This $X$ satisfies the conclusion of the theorem.
\end{proof}

In particular, any fixed order $L$ can appear as an interval in the $X$ constructed in the proof, by simply including $L$ among the $L_j$. 

By an analogous construction, for any $n\geq 1$ we can get an $X$ such that $X \cong X^{\omega}(I_{[u]_n})$ (and hence $X^{n+1} \cong X$), where the nonempty $I_{[u]_n}$ are pairwise non-isomorphic and fill up the classes of every $X_i^{\omega}$ for some increasing sequence $X_0 \subseteq X_1 \subseteq \ldots$ converging to $X$. Although not all of the $n$-tail-equivalence classes $[u]_n$ are filled in the final decomposition $X^{\omega}(I_{[u]_n})$, those that are filled are filled with pairwise non-isomorphic orders. In particular, there will be many examples of where $I_{[u]_n} \not \cong I_{[au]_n}$, and therefore these decompositions will be non-trivial. By doing a longer induction, it is possible to get $X \cong X^{\omega}(I_{[u]_n})$ where every $I_{[u]_n}$ is nonempty. Of course, it is also possible to arrange during the construction that some (or all) of the $I_{[u]_n}$ are isomorphic.

\section{Constructing $A$ and $X$ such that $X \cong A^2X \not\cong AX$}

Our main theorem gives that the cube property holds for the class of linear orders under the lexicographical product. In view of the proof, it is natural to ask if we could have established the stronger result, that for all orders $A$ and $X$, if $A^2X \cong X$ then $AX \cong X$. By \ref{proppra}, if it were possible to construct a p.r.a. for \emph{any} order of the form $A^{\omega}$ then the answer would be yes. This raises the subquestion of whether constructing a p.r.a. for $A^{\omega}$ is always possible. 

We show in this section that the answer to both questions is no. In fact, if $A$ is a complete $(\kappa, \lambda)$-order for cardinals $\kappa$ and $\lambda$ of uncountable cofinality, then $A^{\omega}$ does not admit a p.r.a. Furthermore, the converse to Theorem \ref{proppra} holds: if $A^{\omega}$ does not have a p.r.a., then there exists an $X$ such that $A^2X \cong X$ but $AX \not\cong X$. Thus in particular, there is such an $X$ when $A = \omega_1^* + \omega_1$. In the language of the introduction, this says that the (left-sided) weak Schroeder-Bernstein property fails for the class of linear orders. (The right-sided weak Schroeder-Bernstein property also fails, that is, there exist $X, A$ with $XA^2 \cong X$ but $XA \not \cong X$. This is easier to prove, but we will not do so here.)

Given a linear order $X$, a \emph{cut} in $X$ is a pair of intervals $(I, J)$ such that $I \cup J = X$, $I \cap J = \emptyset$, and $I < J$. Thus $I$ is an initial segment of $X$, and $J$ is a final segment. A cut $(I, J)$ is called a \emph{gap} if $I$ has no maximal element, and $J$ has no minimal element. 

The \emph{Dedekind completion} of $X$, denoted $\overline{X}$, is the order obtained from $X$ by filling every gap $(I, J)$ with a single point $x_{(I, J)}$. Viewing $X$ as a suborder of $\overline{X}$, we have that $X$ is dense in $\overline{X}$, and in particular the cofinality (respectively, coinitiality) of $\overline{X}$ coincides with that of $X$. The Dedekind completion $\overline{X}$ is always a complete linear order, and if $X$ is complete to begin with, then $X = \overline{X}$. 

Any automorphism $f: X \rightarrow X$ can be extended uniquely to an automorphism $\overline{f}: \overline{X} \rightarrow \overline{X}$. For, if $(I, J)$ is a gap in $X$, then $(f[I], f[J])$ must also be a gap. The automorphism $\overline{f}$ is defined by setting $\overline{f}(x_{(I, J)}) = x_{(f[I], f[J])}$ for every gap $(I, J)$, and $\overline{f}(x) = f(x)$ for every $x \in X$.

A subset $C \subseteq X$ is called \emph{closed} if every monotone sequence in $C$ is either unbounded in $X$ or converges to some point in $C$. For $C$ to be closed it is necessary that $C$ is complete in the order inherited from $X$. However, completeness is not sufficient for closure. For example, $\{\frac{1}{n}: n \geq 1\} \cup \{0\}$ is closed as a subset of $\mathbb{R}$, whereas $\{\frac{1}{n}: n \geq 1\} \cup \{-1\}$ is not. Note that if $X$ is not complete, then $X$ is not closed as a subset of itself. 

A subset $C \subseteq X$ is called \emph{left-unbounded} if for every $x \in X$ there exists $c_0 \in C$ such that $c_0 \leq x$, and \emph{right-unbounded} if for every $x$ one can find $c_1 \in C$ with $c_1 \geq x$. If $C$ is unbounded in both directions, we simply say $C$ is \emph{unbounded}. If $C$ is closed and left-unbounded, then $C$ is called a \emph{left club}, and $C$ is a \emph{right club} if it is closed and right-unbounded. If $C$ is both a left and right club, then we say simply that $C$ is a \emph{club}. Similarly, if $I \subseteq X$ is an interval, then viewing $I$ as an order in itself we may speak of a right club in $I$, left club in $I$, and club in $I$. 

It is straightforward to check that if $X$ has uncountable cofinality, then the intersection of two right clubs is a right club, and if $X$ has uncountable coinitiality, then the intersection of two left clubs is a left club. Hence if $X$ has both uncountable cofinality and coinitiality, the intersection of two clubs is a club, the intersection of a club with a right club is a right club, and the intersection of a club with a left club is a left club. 

An order $X$ may not contain any club suborders, but its completion $\overline{X}$ always contains at least one club, namely $\overline{X}$ itself. 

Suppose that $X$ is a $(\kappa, \lambda)$-order, and both $\kappa$ and $\lambda$ have uncountable cofinality (so that $X$ has uncountable cofinality and coinitiality). Let $f$ be an automorphism of $X$, and $\overline{f}$ its extension to $\overline{X}$. Let us check that $C$, the set of fixed points of $\overline{f}$, is a club in $\overline{X}$. It is clear that $C$ is closed. To see that it is unbounded, fix $x \in \overline{X}$. If $\overline{f}(x) = x$, then $x \in C$ and there is nothing to check. So suppose $x \not \in C$. Then either $\overline{f}(x) > x$ or $\overline{f}(x) < x$. Assume without loss of generality that we are in the former case. Then the positive iterates of $x$ under $\overline{f}$ form an increasing sequence, that is, we have 
\[
x < \overline{f}(x) < \overline{f}^2(x) < \ldots
\]
Since $\overline{X}$ has uncountable cofinality, this sequence is bounded, and hence converges (by completeness) to some point $b$. It is easy to see that $b$ must be a fixed point of $\overline{f}$. Symmetrically, since the coinitiality of $\overline{X}$ is also uncountable, the negative iterates of $x$ converge to some $a$, and this $a$ must be fixed by $\overline{f}$. We have found $a, b \in C$ with $a < x < b$, and so $C$ is unbounded as claimed. 

Until further notice, let $A$ denote the order $\omega_1^* + \omega_1$. If $\alpha \in \omega_1$ is an ordinal, we denote the corresponding element in $\omega_1^*$ by $-\alpha$. We identify the $0$ of $\omega_1$ with the $0$ of $\omega_1^*$. Thus
\[
A = \ldots < -\alpha < \ldots < -1 < 0 < 1 < \ldots < \alpha < \ldots
\]

\theoremstyle{definition}
\newtheorem{nopra}[theorem]{Theorem}
\begin{nopra}\label{nopra}
Let $A = \omega_1^* + \omega_1$. Then $A^{\omega}$ does not admit a parity-reversing automorphism. 
\end{nopra}
\begin{proof}
For every finite sequence $r \in A^{<\omega}$, let $I_r$ denote the interval in $A^{\omega}$ consisting of sequences beginning with $r$. We confuse sequences of length 1 with elements of $A$, so that if $a \in A$ then $I_a$ means the interval of points with first entry $a$. Note that $I_r$ is isomorphic to $A^{\omega}$ for every $r$, and in particular has neither a left nor right endpoint. 

While $A$ is complete, $A^{\omega}$ is not, since for every $r$ the interval $I_r$ has a gap to its immediate left and right. To see this, let $J$ be the interval consisting of points in $I_r$ and above, and let $I$ be the complement of $J$. Then $J$ has no minimum, since $I_r$ does not. On the other side, $I$ has no maximum, since any $u \in I$ must begin with a finite sequence $s$, of the same length as $r$, but with some entry $s_i < r_i$. Writing $u = su'$, pick a sequence $u'' > u'$, which is always possible since $A^{\omega}$ has no top point. Let $v = su''$. Then $v > u$ but $v$ is still in $I$ since it lies below $I_r$. Thus the cut $(I, J)$ is in fact a gap, and it lies to the immediate left of $I_r$. Symmetrically (since $A^{\omega}$ has no bottom point), $I_r$ has a gap to the right. 

For every $r \in A^{<\omega}$, let $r^-$ and $r^+$ denote the elements of $\overline{A^{\omega}}$ that fill the gaps to the left and right of $I_r$ respectively. (These points are not pairwise distinct: if $a, b \in A$ and $b = a+1$ then for any $r \in A^{<\omega}$ we have $ra^+ = rb^-$.)

Let $f: A^{\omega} \rightarrow A^{\omega}$ be an automorphism of $A^{\omega}$, and $\overline{f}$ its extension to $\overline{A^{\omega}}$. Let $C \subseteq \overline{A^{\omega}}$ denote the club of fixed points of $\overline{f}$. 

We shall show that $f$ has a fixed point of the form $u = (\alpha_0, -\alpha_1, \alpha_2, -\alpha_3, \ldots)$ for some collection of ordinals $\alpha_i \in \omega_1$. Since such a $u$ cannot be periodic of odd period, we have that $[u]_2 \neq [au]_2$, and thus $f$ is not parity-reversing. 

Consider the $\omega_1$-length increasing sequence of intervals 
\[
I_0 < I_1 < \ldots < I_{\alpha} < \ldots
\]
If $\alpha$ is a limit ordinal, then there are no points in $A^{\omega}$ that lie below $I_{\alpha}$ but above each $I_{\beta}$ for $\beta < \alpha$. This means that the $\omega_1$-sequence of left endpoints
\[
0^- < 1^- < \ldots < \alpha^- < \ldots
\]
is closed in $\overline{A^{\omega}}$. Since this sequence is right unbounded, it forms a right club in $\overline{A^{\omega}}$. Denote this right club by $D_0$.

Then $D_0 \cap C$ is a right club in $\overline{A^{\omega}}$. Fix $\alpha_0^- \in D_0 \cap C$. Let $J_1$ denote the interval of points in $\overline{A^{\omega}}$ lying strictly above $\alpha_0^-$. Then since $\alpha_0^-$ is a fixed point of $\overline{f}$, we have that $\overline{f}$ restricted to $J_1$ is an automorphism of $J_1$. The coinitiality of $J_1$ is $\omega_1$ since $I_{\alpha_0}$ is an initial segment of $J_1$, and its cofinality is also $\omega_1$, since it is a final segment of $A^{\omega}$. Hence $J_1$ is a complete $(\omega_1, \omega_1)$-order. It follows that the set of fixed points of $\overline{f}$ in this interval, $J_1 \cap C$, is a club in $J_1$. 

Now consider the descending sequence of intervals
\[ 
I_{(\alpha_0, 0)} > I_{(\alpha_0, -1)} > \ldots > I_{(\alpha_0, -\alpha)} > \ldots
\]
and the corresponding closed sequence of \emph{right} endpoints
\[
(\alpha_0, 0)^+ > (\alpha_0, -1)^+ > \ldots > (\alpha_0, -\alpha)^+ > \ldots
\]
Denote this sequence by $D_1$. Then $D_1$ is left unbounded in $J_1$, and hence a left club in $J_1$. Thus $D_1 \cap C$ is a left club in $J_1$, and we may fix $(\alpha_0, -\alpha_1)^+ \in D_1 \cap C$. 

At the next stage we define $J_2$ to be the subinterval of $J_1$ consisting of points that lie strictly below $(\alpha_0, -\alpha_1)^+$. Then $\overline{f}$ restricted to $J_2$ is an automorphism of $J_2$. This interval ends with $I_{(\alpha_0, -\alpha_1)}$ and hence has cofinality $\omega_1$. Since it also has coinitiality $\omega_1$, we have again that the set of fixed points in this interval, $J_2 \cap C$, is a club in $J_2$. Hence there must be a fixed point $(\alpha_0, -\alpha_1, \alpha_2)^-$ in the right club sequence 
\[
(\alpha_0, -\alpha_1, 0)^- < (\alpha_0, -\alpha_1, 1)^- < \ldots < (\alpha_0, -\alpha_1, \alpha)^- < \ldots
\]

Continuing in this way, we obtain a sequence of fixed points $(\alpha_0)^-$, $(\alpha_0, -\alpha_1)^+$, $(\alpha_0, -\alpha_1, \alpha_2)^-$, $(\alpha_0, -\alpha_1, \alpha_2, -\alpha_3)^+, \ldots$ While these points all lie outside of $A^{\omega}$, they converge (in the obvious sense) to the point $u = (\alpha_0, -\alpha_1, \alpha_2, \ldots) \in A^{\omega}$. Since the set of fixed points of $\overline{f}$ is closed, it must then be that $u$ is fixed by $\overline{f}$. Since $\overline{f}$ agrees with $f$ on $A^{\omega}$, we have in fact that $u$ is a fixed point of $f$. As observed already, it follows that $f$ is not parity-reversing.
\end{proof}

It is easy to generalize the proof to get that if $A$ is any complete $(\kappa, \lambda)$-order, where $\kappa$ and $\lambda$ have uncountable cofinality, then $A^{\omega}$ does not have a parity-reversing automorphism.

We will now prove the converse to \ref{proppra}. Recall that a linear order is called \emph{scattered} if it does not contain an infinite, dense suborder. In particular, every ordinal $\alpha$, considered as a linear order, is scattered. 

\theoremstyle{definition}
\newtheorem{AXnotX}[theorem]{Theorem}
\begin{AXnotX}\label{AXnotX}
Suppose that $A^{\omega}$ does not have a parity-reversing automorphism. Then there exists an order $X$ such that $A^2 X \cong X$ but $AX \not \cong X$. 
\end{AXnotX}
\begin{proof}
Since $A^{\omega}$ does not have a p.r.a., it must be that $A$ either has no endpoints, or only a single endpoint. In either case $A^{\omega}$ is dense. Hence any interval in $A^{\omega}$, considered as a linear order itself, is dense. Suppose that $A^{\omega}(I_u)$ is any replacement of $A^{\omega}$ such that for densely many $u$ we have $I_u \neq \emptyset$. Then since $A^{\omega}$ is dense, for any interval $I \subseteq A^{\omega}(I_u)$, we have that either $I \subseteq I_u$ for some $u$, or $I$ contains a dense suborder. (The ``or" here is non-exclusive: both conditions may hold, but not neither.) In the latter case, $I$ is by definition non-scattered. 

Let $\kappa$ denote the number of $\sim_2$-equivalence classes in $A^{\omega}$. Enumerate these classes as $\{C_{\alpha}: \alpha \in \kappa \}$. Let $X$ be the replacement of $A^{\omega}$ obtained by replacing every point in the $\alpha$th class with a copy of $\alpha$, that is, $X = A^{\omega}(I_{[u]_2})$ where $I_{[u]_2} = \alpha$ if $[u]_2 = C_{\alpha}$. (It is inessential, though convenient, that we replace the points with ordinals. Any $\kappa$-sized collection of pairwise non-isomorphic scattered orders also works to define the $I_{[u]_2}$.)

By construction we have that $A^2X \cong X$. We know from \ref{switching} that $AX \cong A^{\omega}(J_{[u]_2})$ where $J_{[u]_2} = I_{[au]_2}$ for all $u$ and $a$. Suppose toward a contradiction that there is an isomorphism $f$ of $X$ with $AX$, that is, of $A^{\omega}(I_{[u]_2})$ with $A^{\omega}(J_{[u]_2})$. 

For a fixed $u$, the order type of the interval $I_u$ is an ordinal $\alpha$, and in particular is scattered. Its image $f[I_u]$ must be an interval of type $\alpha$ in $A^{\omega}(J_{[u]_2})$, and so from our observation above, it must be that $f[I_u] \subseteq J_v$ for some $v$. By the same argument, $f^{-1}[J_v]$ must be contained in $I_w$ for some $w$. But then since $f^{-1}[J_v]$ contains $I_u$, it must be that in fact $v = w$, and $f[I_u] = J_v$. But then $J_v = \alpha$, and hence it must be that $v \in [au]_2$. Moreover, $f$ must be the identity on $I_u$, since the identity is the only isomorphism of $\alpha$ with itself. 

Thus the isomorphism $f$ may be factored as $(g, \textrm{id})$, where $g: A^{\omega} \rightarrow A^{\omega}$ is a parity-reversing automorphism. But no such $g$ exists, by hypothesis. Hence there is no isomorphism between $X$ and $AX$. 
\end{proof}

\section{Related Problems}

In \emph{Cardinal and Ordinal Numbers}, Sierpi\'nski poses several other questions concerning the multiplication of linear orders aside from the cube problem. On page 232 he writes, ``We do not know so far any example of two types $\varphi$ and $\psi$, such that $\varphi^2 = \psi^2$ but $\varphi^3 \neq \psi^3$, or types $\gamma$ and $\delta$ such that $\gamma^2 \neq \delta^2$ but $\gamma^3 = \delta^3$." Later, on page 251, ``We do not know whether there exist two different denumerable order types which are left-hand divisors of each other. Neither do we know whether there exist two different order types which are both left-hand and right-hand divisors of each other."

Since Sierpi\'nski ordered products anti-lexicographically, ``left-hand divisor" for him means ``right-hand divisor" in the convention of this paper. Writing them out using our convention, his questions are,
\begin{enumerate}
\item[1.] Do there exist orders $X$ and $Y$ such that $X^2 \cong Y^2$ but $X^3 \not\cong Y^3$?
\item[2.] Do there exist orders $X$ and $Y$ such that $X^2 \not\cong Y^2$, but $X^3 \cong Y^3$?
\item[3.] Do there exist \emph{countable} orders $X$ and $Y$ such that $X \not \cong Y$ but for some orders $A, B$ we have $AY \cong X$ and $BX \cong Y$ ?
\item[4.] Do there exist orders $X$ and $Y$ such that $X \not\cong Y$ but for some orders $A_0, A_1, B_0, B_1$ we have $ A_0Y \cong YA_1 \cong X$ and $B_0X \cong XB_1 \cong Y$?
\end{enumerate}

In comparison with the questions from the introduction, these questions are phrased negatively, asking for counterexamples to the corresponding properties. 

Sierpi\'nski was aware of counterexamples to the unique square root property for linear orders, that is, of non-isomorphic orders $X$ and $Y$ such that $X^2 \cong Y^2$. These examples are due to Morel; see \cite{Morel}. Question 2 is the natural generalization of the unique square root problem, and could be called the unique cube root problem. More generally, one may ask if there exist orders $X$ and $Y$ such that $X^n \cong Y^n$ but $X^k \not\cong Y^k$ for $k<n$. 

Question 1 is motivated by the fact that Morel's examples $X, Y$ of non-isomorphic orders with isomorphic squares have the property that $X^n \cong X$ for all $n \geq 1$ and $Y^n \cong X$ for all $n > 1$. In particular, not only is it the case that $X^2 \cong Y^2$  but actually that $X^n \cong Y^n$ for all $n>1$. Question 1 asks whether this kind of collapsing is necessary, or if it is possible that two orders have isomorphic squares but non-isomorphic cubes.

Both Questions 1 and 2 are related to a generalization of the cube problem. Suppose it were possible to find an order $X$ such that $X^5 \cong X$ but the powers $X, X^2, X^3$, and $X^4$ were pairwise non-isomorphic. Then $X$ and $Y=X^3$ would give a positive answer to Question 2. Similarly, if it were possible to find an $X$ isomorphic to $X^7$ but whose intermediate powers were pairwise non-isomorphic, then $X$ and $Y=X^3$ would give a positive answer to Question 1. By our main theorem, there are no such $X$, but it may still be that these two questions have positive answers. These questions are, to the author's knowledge, still open.

Question 4 might be called the \emph{two-sided Schroeder-Bernstein problem} for the class $(LO, \times)$. It is a sensible problem to ask given that the lexicographical product is non-commutative and that there exist examples witnessing the failure of left-sided Schroeder-Bernstein property and right-sided Schroeder-Bernstein property for linear orders. It is closely related to the cube problem, in the sense that if there existed an $X$ isomorphic to its cube but not its square, then $X$ and $Y=X^2$ would give a positive answer. There is no such $X$, but it turns out that Question 4 still has a positive answer. The proof of this will appear separately. This gives further evidence that the cube property for $(LO, \times)$ is ``close" to being false. 

Question 3 is the left-sided Schroeder-Bernstein problem for \emph{countable} linear orders. Sierpi\'nski was aware of \emph{uncountable} orders $X, Y, A, B$ satisfying the relations in the problem. In Section 6 we constructed such orders with $A = B$. From the discussion following \ref{cntblwidth}, we know that in the case when $A=B$, any orders $X$ and $Y$ satisfying the relations of Question 3 \emph{must} be uncountable. We will now strengthen this result, and show that Question 3 has a negative answer.

Suppose that $X, Y, A, B$ are countable orders, and $AY \cong X$ and $BX \cong Y$. We prove $X \cong Y$. Notice first that the hypotheses give $ABX \cong X$. Let $C = AB$. There are three cases. If $C$ has both endpoints, then both $A$ and $B$ must also have both endpoints. But then $X$, which is isomorphic to $AY$, contains an initial copy of $Y$, and $Y$, which is isomorphic to $BX$, contains a final copy of $X$. Hence $X \cong Y$ by Lindenbaum's theorem. If $C$ has neither endpoint, then by the proof of \ref{cntblnoend}, since $CX \cong X$ we have that $X$ must be of the form $\mathbb{Q}(I_i)$. Thus $X$ is invariant under left multiplication by any countable order. In particular, $BX \cong X$, that is, $Y \cong X$. Finally, suppose $C$ has a single endpoint. Without loss of generality, assume it is the left endpoint. Then both $A$ and $B$ have a left endpoint (and at least one of them is missing the right endpoint). By the adaption of the proof of \ref{cntblnoend} to the left endpoint case, $X$ must be of the form $Q(I_i)$, where $Q = \mathbb{Q} \cap [0, 1)$, and where the order replacing the left endpoint $0$ also has a left endpoint. It can be shown that such an order is invariant under left multiplication by any countable order with a left endpoint. In particular, $BX \cong X$, that is, $Y \cong X$.

All of Sierpi\'nski's questions are instances of a much more general question. If we distinguish the structures in a given class $\mathfrak{C}$ only up to isomorphism type, then $(\mathfrak{C}, \times)$ may be viewed as a (possibly very large) semigroup, where the semigroup operation is given by the product. For a given semigroup $(S, \cdot)$, we say that $S$ can be \emph{represented} in $\mathfrak{C}$ if there is a map $\iota: S \rightarrow \mathfrak{C}$ such that $\iota(a\cdot b) \cong \iota(a) \times \iota(b)$, and if $a \neq b$ then $\iota(a) \not\cong \iota(b)$.

The failure of the cube property for a given class $\mathfrak{C}$ is equivalent to the statement that the group $\mathbb{Z}_2$ can be represented in $\mathfrak{C}$. Thus our main theorem gives that $\mathbb{Z}_2$, and more generally $\mathbb{Z}_n$, cannot be represented in $(LO, \times)$. 

\begin{enumerate}
\item[5.] Which semigroups can be represented in $(LO, \times)$?
\item[6.] Can any nontrivial group be represented in $(LO, \times)$?
\end{enumerate}

Question 6 is of course a subquestion of Question 5. By our results, if Question 6 has a positive answer, then any non-identity element in the witnessing group must be of infinite order.

Questions 1, 2, and 4 all concern relations that can be realized in certain semigroups. Thus a complete answer to Question 5 would yield answers to all of these questions.

\end{document}